\documentclass{amsart}

\usepackage{amssymb}
\usepackage{amsmath}
\usepackage{amsthm}
\usepackage[mathscr]{eucal}

\newtheorem{thrm}{Theorem}[section]
\newtheorem{lemma}[thrm]{Lemma}
\newtheorem{prop}[thrm]{Proposition}
\newtheorem{cor}[thrm]{Corollary}

\theoremstyle{definition}
\newtheorem{defn}[thrm]{Definition}

\theoremstyle{remark}

\numberwithin{equation}{section}

\newcommand{\dbar}{$\bar{\partial}$}
\newcommand{\mdbar}{\bar{\partial}}

\newcommand{\lre}{\mathscr{E}}
\newcommand{\lra}{\mathscr{A}}
\newcommand{\lrm}{\mathcal{M}}
\newcommand{\lrl}{\mathcal{L}}
\newcommand{\lrq}{\mathcal{Q}}

\newcommand{\lrt}{\mathcal{T}}
\newcommand{\lrp}{\mathcal{P}}

\begin{document}

\bibliographystyle{plain}

\title{Integral representations on non-smooth domains}

\author{Dariush Ehsani}

\address{Department of Mathematics, Penn State - Lehigh Valley, Fogelsville,
PA 18051}
 \email{ehsani@psu.edu}
\curraddr{Humboldt-Universit\"{a}t, Institut f\"{u}r Mathematik,
10099 Berlin }


\thanks{Partially supported by the Alexander von Humboldt Stiftung}

\begin{abstract}  We derive integral representations for
$(0,q)$-forms, $q\ge1$, on non-smooth strictly pseudoconvex
domains, the Henkin-Leiterer domains.  A $(0,q)$-form, $f$ is
written in terms of integral operators acting on $f$, $\mdbar f$,
and $\mdbar^{\ast} f$.  The representation is applied to derive
$L^{\infty}$ estimates.
\end{abstract}

\maketitle

\ \\
\section{Introduction}
Lieb and Range in \cite{LR86} developed a powerful integral
representation by which estimates in the theory of the
\dbar-Neumann problem could be deduced.  The main theorem was an
integral representation of $(0,q)$-forms on $D\subset\subset X$ a
smooth strictly pseudoconvex domain in a complex manifold $X$.
\begin{thrm}[Lieb-Range]
\label{lrthrm}
 Let $P_0:L^2(D)\rightarrow
\mathcal{O}\cap L^2(D)$ be the Bergman projection.  There exist
integral operators $T_q:L^2_{(0,q+1)}(D)\rightarrow
L^2_{(0,q)}(D)$ $0\le q<n=\dim X$ such that for $f\in
L^2_{(0,q)}\cap Dom(\bar{\partial})\cap
Dom(\bar{\partial}^{\ast})$ one has
\begin{equation*}
f=P_0f+T_0\bar{\partial}f +\mbox{ error terms } \quad \mbox{ for }
q=0 \end{equation*}
 and
\begin{equation}
 \label{highorder}
f=T_q\bar{\partial}f+T_{q-1}^{\ast}\bar{\partial}^{\ast}f +\mbox{
error terms } \quad \mbox{ for } q\ge 1.
\end{equation}
\end{thrm}

 In
(\ref{highorder}) the metric has to be carefully adapted to the
boundary.  The choice of the metric as the Levi metric as in
Greiner and Stein \cite{GS} was essential in their "cancellation
of singularities" argument, which allowed for treatment of terms
in the representation as error terms.

We take up the problem here of establishing an integral
representation in the manner of \cite{LR86} relaxing the
assumption that $D$ be smooth.
  Let $D$ have a defining function, $r$.  We allow for singularities in
the boundary, $\partial D$ of $D$ by permitting the possibility
that $dr$ vanishes at points on $\partial D$.  Such domains were
first studied by Henkin and Leiterer in \cite{HL}, and we
therefore refer to them as Henkin-Leiterer domains.

  We shall make the additional assumtion that $r$ is a Morse
function.  Let $U$ be a neighborhood of $\partial D$.  Then
\begin{equation*}
U\cap D=\{ x\in U: r(x)<0\},
\end{equation*}
 $r$ with only non-degenerate critical points on $U$.  We have
\begin{equation*}
\partial D=\{ x:r(x)=0\},
\end{equation*}
and we can assume that there are finitely many critical points on
$bD$, and none on $U\setminus bD$.

In \cite{EhLi}, Lieb and the author studied the Bergman projection
on Henkin-Leiterer domains in $\mathbb{C}^n$, and obtained
weighted $L^p$ estimates.  We here concern ourselves with proving
an analogue of (\ref{highorder}) on Henkin-Leiterer domains.  The
domain $D$ has an exhaustion of smooth strictly pseudoconvex
domains $\{D_{\epsilon}\}_{\epsilon}$ on each of which the
analysis of Lieb and Range applies.  One immediate problem one
runs into with this approach is that forms which are
Dom$(\mdbar^{\ast})$ on $D$ are may not be in
Dom$(\mdbar^{\ast}_{\epsilon})$ on $D_{\epsilon}$.  We deal with
this problem by using a density lemma of Henkin and Iordan
\cite{HI} which provides for forms $f_{\epsilon}$ which are in
$L^2(D_{\epsilon})\cap \mbox{Dom}(\mdbar)\cap
\mbox{Dom}(\mdbar^{\ast}_{\epsilon})$ and which approximate a
given $f\in L^2(D)\cap \mbox{Dom}(\mdbar)\cap
\mbox{Dom}(\mdbar^{\ast})$.
 Our approach therefore is to obtain an
integral representation valid on each domain $D_{\epsilon}$ and in
the end let $\epsilon\rightarrow 0$.
 In this approach we need to multiply our operators by factors of
  $|dr|$ so that convergence of the representation as
 $\epsilon\rightarrow 0$ is obtained.  Let $\gamma=|\partial r|$.
 The analogue of Theorem
\ref{lrthrm} we establish here is
\begin{thrm}  Let $D$ be a Henkin-Leiterer domain with a defining
function which is Morse.  There exist integral operators
$\tilde{T}_q:L^2_{(0,q+1)}(D)\rightarrow L^2_{(0,q)}(D)$ $0\le
q<n=\dim X$ such that for $f\in L^2_{(0,q)}\cap
Dom(\bar{\partial})\cap Dom(\bar{\partial}^{\ast})$ one has
\begin{equation*}
\gamma^3
 f=\tilde{T}_q\bar{\partial}f+\tilde{T}_{q-1}^{\ast}\bar{\partial}^{\ast}f
+\mbox{ error terms } \quad \mbox{ for } q\ge 1.
\end{equation*}
\end{thrm}

In a separate paper we build off the integral representation
established here, and in particular we look at the mapping
properties of the integral operators under differentiation so as
to establish $C^k$ estimates.

The author wishes to acknowledge the fruitful discussions with
Ingo Lieb over the matters in this paper.  His ideas and advice on
particulars were instrumental in achieving the results here.

\section{Admissible operators}

   With local coordinates denoted by $\zeta_1,\ldots,\zeta_n$, we
define a Levi metric in a neighborhood of $\partial D$ by
\begin{equation*}
ds^2= \sum_{j,k} \frac{\partial^2 r}{\partial \zeta_j, \partial
\overline{\zeta}_k} (\zeta)d\zeta_j d\bar{\zeta}_k.
\end{equation*}
A Levi metric on $X$ is a Hermitian metric which is a Levi metric
in a neighborhood of $\partial D$.

We thus equip $X$ with a Levi metric and we take $\rho(x,y)$ to be
a symmetric, smooth function on $X\times X$ which coincides with
the geodesic distance in a neighborhood of the diagonal,
$\Lambda$, and is positive outside of $\Lambda$.

For ease of notation, in what follows we will always work with
local coordinates, $\zeta$ and $z$.

      Since $D$ is strictly pseudoconvex and $r$ is a Morse function, we
can take $r_{\epsilon}=r+\epsilon$ for epsilon small enough.  Then
$r_{\epsilon}$ will be defining functions for smooth, strictly
pseudoconvex $D_{\epsilon}$.  For such $r_{\epsilon}$ we have that
all derivatives of $r_{\epsilon}$ are indpendent of $\epsilon$.
 In particular, $\gamma_{\epsilon}(\zeta)=\gamma(\zeta)$ and
$\rho_{\epsilon}(\zeta,z)=\rho(\zeta,z)$.

 Let $F$ be the Levi
polynomial for $D_{\epsilon}$:
\begin{equation*}
F(\zeta,z)
 = \sum_{j=1}^n\frac{\partial
 r_{\epsilon}}{\partial\zeta_j}(\zeta)(\zeta_j-z_j)
  -\frac{1}{2}\sum_{j,k=1}^n\frac{\partial^2
  r_{\epsilon}}{\partial\zeta_j\partial\zeta_k}(\zeta_j-z_j)(\zeta_k-z_k).
\end{equation*}
 We note that $F(\zeta,z)$ is independent of $\epsilon$ since
 derivatives of $r_{\epsilon}$ are.

For $\epsilon$ small enough we can choose $\delta>0$ and
$\varepsilon>0$ and a patching function $\varphi(\zeta,z)$,
independent of $\epsilon$, on $\mathbb{C}^n\times\mathbb{C}^n$
such that
\begin{equation*}
\varphi(\zeta,z)=
 \begin{cases}
1 & \mbox{for } \rho^2(\zeta,z)\le
\frac{\varepsilon}{2}\\
0 &  \mbox{for } \rho^2(\zeta,z)\ge \frac{3}{4}\varepsilon,
\end{cases}
\end{equation*}
and defining $S_{\delta}=\{ \zeta:|r(\zeta)|<\delta\}$,
$D_{-\delta}=\{ \zeta:r(\zeta)<\delta\}$, and
\begin{equation*}
\phi_{\epsilon}(\zeta,z)
 =\varphi(\zeta,z)(F(\zeta,z)-r_{\epsilon}(\zeta))
  +(1-\varphi(\zeta,z))\rho^2(\zeta,z),
 \end{equation*}
we have the following
\begin{lemma}
 \label{phiest}
On $D_{\epsilon}\times D_{\epsilon}\bigcap S_{\delta}\times
D_{-\delta}$,
 \begin{equation*}
|\phi_{\epsilon}|\gtrsim |\langle \partial
r_{\epsilon}(z),\zeta-z\rangle|+\rho^2(\zeta,z),
\end{equation*}
where the constants in the inequalities are independent of
$\epsilon$.
\end{lemma}
\begin{proof}
From a Taylor series expansion
\begin{equation}
\label{phiestregprop}
 |\phi_{\epsilon}|\gtrsim
 -r_{\epsilon}(\zeta)-r_{\epsilon}(z)
+\rho^2(\zeta,z)+|\mbox{Im}\phi_{\epsilon}|.
\end{equation}
On $D_{\epsilon}\times D_{\epsilon}$,
$-r_{\epsilon}(\zeta)-r_{\epsilon}(z)
 \ge |r_{\epsilon}(\zeta)-r_{\epsilon}(z)|$.
We combine this with
\begin{equation*}
|\mbox{Im}\phi_{\epsilon}|+\rho^2(\zeta,z)\gtrsim
|\mbox{Im}\langle
\partial
 r_{\epsilon}(\zeta),\zeta -z\rangle|,
\end{equation*}
%
%
and we therefore write
\begin{align*}
|\phi_{\epsilon}|&\gtrsim
 |r_{\epsilon}(\zeta)-r_{\epsilon}(z)|+|\mbox{Im} \langle \partial
r_{\epsilon}(\zeta),\zeta -z\rangle| +\rho^2(\zeta,z)\\
 &\gtrsim |\langle \partial
r_{\epsilon}(z),\zeta-z\rangle|+\rho^2(\zeta,z),
\end{align*}
where the last inequality follows from
\begin{multline*}
|\langle \partial r_{\epsilon}(z),\zeta-z\rangle|+|\langle
\partial r_{\epsilon}(\zeta),\zeta-z\rangle|+\rho^2(\zeta,z)
 \approx \\
 |r_{\epsilon}(\zeta)-r_{\epsilon}(z)|+|\mbox{Im} \langle \partial
r_{\epsilon}(\zeta),\zeta -z\rangle| +\rho^2(\zeta,z),
\end{multline*}
which itself is an easy consequence of a Taylor expansion.

All inequality signs have constants which are independent of
$\epsilon$ since $r_{\epsilon}\overset{C^2}\rightarrow r$.
\end{proof}

 We at times have to be precise and keep track of factors
 of $\gamma$ which occur in our integral kernels.  We shall write
$\lre_{j,k}(\zeta,z)$ for those double forms on open sets
$U\subset D\times D$ such that $\lre_{j,k}$ is smooth on $U$ and
satisfies
\begin{equation}
 \label{defnxi}
\lre_{j,k}(\zeta,z)\lesssim \xi_k(\zeta) |\zeta-z|^j,
\end{equation}
 where
 $\xi_k$ is a smooth function in $D$ with the property
\begin{equation*}
|\gamma^{\alpha}D_{\alpha}\xi_k|\lesssim \gamma^k,
\end{equation*}
for $D_{\alpha}$ a differential operator of order $\alpha$.

 We shall write
 $\lre_j$ for those double forms on open sets
$U\subset D\times D$ such that $\lre_{j}$ is smooth on $U$, can be
extended smoothly to $\overline{D}\times\overline{D}$, and
satisfies
\begin{equation*}
\lre_{j}(\zeta,z)\lesssim  |\zeta-z|^j.
\end{equation*}
$\lre_{j,k}^{\ast}$ will
 denote
 forms which can be written as $\lre_{j,k}(z,\zeta)$.

For $N\ge 0$, we let $R_N$ denote an $N$-fold product, or a sum of
such products, of first derivatives of $r(z)$, with the notation
$R_0=1$.

 Here
\begin{equation*}
P_{\epsilon}(\zeta,z)=\rho^2(\zeta,z)+
2\frac{r_{\epsilon}(\zeta)}{\gamma(\zeta)}\frac{r_{\epsilon}(z)}{\gamma(z)}.
\end{equation*}

\begin{defn} A double differential form $\lra^{\epsilon}(\zeta,z)$ on
$\overline{D}_{\epsilon}\times\overline{D}_{\epsilon}$ is an
\textit{admissible} kernel, if it has the following properties:
\begin{enumerate}
\item[i)] $\lra^{\epsilon}$ is smooth on
$\overline{D}_{\epsilon}\times\overline{D}_{\epsilon}-\Lambda_{\epsilon}$
 \item[ii)] For each point $(\zeta_0,\zeta_0)\in \Lambda_{\epsilon}$ there is
 a neighborhood $U\times U$ of $(\zeta_0,\zeta_0)$ on which $\lra^{\epsilon}$ or $\overline{\lra}^{\epsilon}$
 has the representation
 \begin{equation}
 \label{typerep}
  R_NR_M^* \lre_{j,\alpha}\lre_{k,\beta}^{\ast}
  P^{-t_0}_{\epsilon}\phi^{t_1}_{\epsilon}\overline{\phi}^{t_2}_{\epsilon}
   \phi^{\ast t_3}_{\epsilon}\overline{\phi}^{\ast t_4}_{\epsilon} r^l_{\epsilon} r^{\ast
   m}_{\epsilon}
 \end{equation}
with $N,M,\alpha,\beta,j,k, t_0, \ldots, m$ integers and $j,k,
t_0, l, m \ge 0$,
 $-t=t_1+\cdots+t_4\le 0$, $N, M\ge
0$, and $N+\alpha, M+\beta\ge 0$.
\end{enumerate}
The above representation is of \textit{smooth type} $s$ for
\begin{equation*}
s=2n+j+\min\{2, t-l-m\} -2(t_0+t-l-m).
\end{equation*}
We define the \textit{type} of $\lra^{\epsilon}(\zeta,z)$ to be
\begin{equation*}
\tau=s-\max \{ 0,2-N-M-\alpha-\beta\}.
\end{equation*}
  $\lra^{\epsilon}$ has \textit{smooth
type} $\ge s$ if at each point $(\zeta_0,\zeta_0)$ there is a
representation (\ref{typerep}) of smooth type $\ge s$.
$\lra^{\epsilon}$ has \textit{type} $\ge \tau$ if at each point
$(\zeta_0,\zeta_0)$ there is a representation (\ref{typerep}) of
type $\ge \tau$.  We shall also refer to the \textit{double type}
of an operator $(\tau,s)$ if the operator is of type $\tau$ and of
smooth type $s$.
\end{defn}
The definition of smooth type above is taken from \cite{LR86}.
Here and below $(r_{\epsilon}(x))^{\ast }=r_{\epsilon}(y)$, the
$\ast$ having a similar meaning for other functions of one
variable.

Let $\lra_{j}^{\epsilon}$ be kernels of type $j$.  We denote by
$\lra_j$ the pointwise limit as $\epsilon\rightarrow 0$ of
$\lra_j^{\epsilon}$ and define the double type of $\lra_j$ to be
the double type of the $\lra_j^{\epsilon}$ of which it is a limit.
We also denote by $A_j^{\epsilon}$ to be operators with kernels of
the form $\lra_j^{\epsilon}$.  $A_j$ will denote the operators
with kernels $\lra_j$.  We use the notation
$\lra_{(j,k)}^{\epsilon}$ (resp. $\lra_{(j,k)}$) to denote kernels
of double type $(j,k)$.

 We begin with estimates on the kernels of a certain type.
\begin{prop}
\label{typicalest}
 Let $\lra^{\epsilon}_j$ be of type j,
  and
\begin{equation*}
 1\le \lambda<\frac{2n+2}{2n+2-j}.
\end{equation*}
Then
\begin{equation}
\label{zetaint}
 \int_{D_{\epsilon}}|\lra_j^{\epsilon}(\zeta,z)|^{\lambda}
dV(\zeta) < C
\end{equation}
 and, similarly,
\begin{equation}
 \label{zint}
\int_{D_{\epsilon}}|\lra_j^{\epsilon}(\zeta,z)|^{\lambda} dV(z) <
C
\end{equation}
for $C<\infty$ a constant independent of $\epsilon$, $z$ or
$\zeta$.
\end{prop}
\begin{proof}
That (\ref{zetaint}) and (\ref{zint}) hold for a fixed
$\epsilon>0$ and a constant $C$ which may depend on $\epsilon$
follows from the results on smooth strictly pseudoconvex domains
(see \cite{LiMi}).  We will perform the calculations in the limit
$\epsilon\rightarrow 0$ so that standard uniform boundedness
principles apply to provide bounds uniform in $\epsilon$.

We handle the estimates case by case depending on the kernel's
double type.  For the various cases we now describe the coordinate
system with which we work.  Fix $z$ such that $\gamma(z)\ne 0$.
 We define the complex tangent space at $z$:
\begin{equation*}
T_z^c=\{ \zeta:\langle
\partial
 r(z),\zeta -z\rangle=0\}.
\end{equation*}
We define the orthonormal system of coordinates, $s_1$, $s_2$,
$t_1, \ldots, t_{2n-2}$ such that
\begin{align*}
s_1&= \mbox{Re } \left\langle \frac{\partial r(z)}{\gamma(z)},\zeta-z\right\rangle\\
s_2&= \mbox{Im } \left\langle \frac{\partial
r(z)}{\gamma(z)},\zeta-z\right\rangle,
\end{align*}
and such that $t_1,\ldots,t_{2n-2}$ span $T_z^{c\bot}$.  Let also
\begin{align*}
&s=\sqrt{s_1^2+s_2^2}\\
&t=\sqrt{t_1^2+\cdots+t_{2n-2}^2}.
\end{align*}
 From Lemma \ref{phiest} we have
\begin{equation*}
|\phi| \gtrsim |\langle \partial r(z),\zeta-z\rangle|+\rho^2,
\end{equation*}
which in the above coordinates reads
\begin{equation*}
|\phi|\gtrsim \gamma(z)s+s^2+t^2.
\end{equation*}
\\
 $Case\ a)$.  $\lra_j$ is of double type $(j,j)$.

For kernels of double type $(j,j)$,
 we can use the relation $\gamma(\zeta)=\gamma(z)+\lre_{1,0}$ along with estimates for kernels of double
 types $(j,j+1)$ and $(j,j+2)$, to reduce the different subcases we need to consider
 to
\begin{align*}
i)&\quad |\lra_j| \lesssim \frac{\gamma(z)^2}{P^{n-j/2}}\\
ii)&\quad |\lra_j| \lesssim
 \frac{\gamma(z)^2}{P^{n-\frac{j+1}{2}}|\phi|}\\
iii)&\quad |\lra_j| \lesssim
\frac{\gamma(z)^2}{P^{n-j/2-\mu}|\phi|^{\mu+1}} \quad \mu\ge 1.
\end{align*}

We will consider the last two subcases, since the first is easier
to
handle, and can be covered by case $c)$ below.\\
$Subcase\ ii).$

We choose $\alpha<2$ such that
\begin{equation*}
\lambda<\frac{2n-2+2\alpha}{2n+1-j},
\end{equation*}
and let $\beta=\min(\alpha, \lambda)$.
 We have
\begin{align}
\nonumber
 \int_{D} &\gamma(z)^{2\lambda}\frac{1}
  {|\phi|^{\lambda}P^{\lambda(n-\frac{j+1}{2})}}dV(\zeta)\\
  \nonumber
&\lesssim \gamma(z)^{2\lambda}\int_V
 \frac{st^{2n-3}}{(\gamma(z) s+s^2+t^2)^{\lambda}(s^2+t^2)^{\lambda(n-\frac{j+1}{2})}}dsdt\\
 \label{beta}
&\lesssim \gamma(z)^{2\lambda-\beta}\int_V
s^{1-\beta}\frac{t^{2n-3}}{(s+t)^{2\lambda n+\lambda(1-j)-2\beta}}
ds dt.
\end{align}
where $V$ is a bounded subset of $\mathbb{R}^2$.
 In the case $\beta=\alpha$ we can estimate the integral in
 (\ref{beta}) by
\begin{equation*}
 \gamma(z)^{2\lambda-\alpha}\int_V
s^{1-\alpha}\frac{t^{2n-3}}{t^{2\lambda n+\lambda(1-j)-2\alpha}}
ds dt \lesssim 1,
\end{equation*}
where the inequality follows from our choice of $\alpha$.

In the case that $\beta=\lambda$ we choose a $\sigma$ such that
\begin{align*}
& \sigma<2-\lambda\\
& \lambda < \frac{2n-2+\sigma}{2n-1-j},
\end{align*}
and we have
\begin{align*}
\gamma(z)^{\lambda}\int_V
s^{1-\lambda}&\frac{t^{2n-3}}{(s+t)^{\lambda(2n-1-j)}} ds dt\\
 &\lesssim
\gamma(z)^{\lambda}\int_V
s^{1-\lambda-\sigma}\frac{t^{2n-3}}{t^{\lambda(2n-1-j)-\sigma}} ds
dt\\
&\lesssim 1,
\end{align*}
where the last inequality follows from our choice of $\sigma$.

(\ref{zint}) holds in a similar manner by switching $\zeta$ and
$z$.
\\
$Subcase\ iii).$  In this case to prove (\ref{zetaint}) we choose
$\alpha$ so that $\alpha<2$ and
\begin{equation*}
\lambda<\frac{2n-2+2\alpha}{2n+2-j},
\end{equation*}
and estimate
\begin{align*}
\gamma(z)^{2\lambda}&\int_{D}\frac{1}
  {|\phi|^{\lambda(\mu+1)}P^{\lambda(n-\mu-j/2)}}dV(\zeta)\\
   &\lesssim
\gamma(z)^{2\lambda}\int_V
 \frac{st^{2n-3}}{(\gamma(z)s+s^2+t^2)^{\lambda(\mu+1)}(s^2+t^2)^{\lambda(n-\mu-j/2)}}dsdt\\
&\lesssim
 \gamma(z)^{2\lambda-\alpha} \int_V
 s^{1-\alpha}\frac{t^{2n-3}}{t^{\lambda(2n+2-j)-2\alpha}}dtds\\
&\lesssim
 1,
\end{align*}
where $V$ is a bounded subset of $\mathbb{R}^2$.

Again, (\ref{zint}) holds in a similar manner.\\
$Case\ b)$.  $\lra_j$ is of double type $(j,j+1)$.

The different subcases we need to consider are
\begin{align*}
i)&\quad |\lra_j| \lesssim \frac{\gamma(z)}{P^{n-\frac{j+1}{2}}}\\
ii)&\quad |\lra_j| \lesssim
 \frac{\gamma(z)}{P^{n-\frac{j+2}{2}}|\phi|}\\
iii)&\quad |\lra_j| \lesssim
\frac{\gamma(z)}{P^{n-\frac{j+1}{2}-\mu}|\phi|^{\mu+1}} \quad
\mu\ge 1.
\end{align*}
Subcases $i)$ and $ii)$ can be handled by the estimate in case
$c)$ below. The more difficult estimate is that of subcase $iii)$,
for which we choose an $\alpha<1/2$ which satisfies
\begin{equation*}
\lambda < \frac{2n+2\alpha}{2n+1-j}
\end{equation*}
 and estimate
\begin{align*}
\gamma(z)^{\lambda}&\int_{D}\frac{1}
  {|\phi|^{\lambda(\mu+1)}P^{\lambda(n-\frac{j+1}{2}-\mu)}}dV(\zeta)\\
   &\lesssim
\gamma(z)^{\lambda}\int_V
 \frac{st^{2n-3}}{(\gamma(z)s+s^2+t^2)^{\lambda(\mu+1)}(s^2+t^2)^{\lambda(n-\frac{j+1}{2}-\mu)}}dsdt\\
&\lesssim \gamma(z)^{\lambda-1}
  \int_V
 \frac{t^{2n-3}}{(s^2+t^2)^{\lambda(n-\frac{j-1}{2})-1}}dtds\\
&\lesssim
 \int_V s^{-2\alpha}
 t^{2n-1-\lambda(2n+1-j)+2\alpha}dtds\\
 &\lesssim 1,
\end{align*}
where $V$ is a bounded subset of $\mathbb{R}^2$.
\\
$Case\ c)$.  $\lra_j$ is of double type $(j,j+2)$.

  Using the coordinates of cases $a)$ and $b)$, we can estimate all the subcases for kernels of double type $(j,j+2)$
by
\begin{align*}
\int_{D} |\lra_j(\zeta,z)|^{\lambda}dV(\zeta) &\lesssim
 \int_V
 \frac{st^{2n-3}}{(s^2+t^2)^{\lambda(n-j/2)}}dsdt\\
 &\lesssim \int_0^M r^{2n-1-\lambda(2n-j)} dr\\
 &\lesssim 1,
\end{align*}
where $V$ is a bounded subset of $\mathbb{R}^2$, $M>0$ is a
bounded constant, and $r=\sqrt{s^2+t^2}$.

The same estimates hold for (\ref{zint}).
\end{proof}
As a consequence of Proposition \ref{typicalest} and a
generalization of Young's inequality \cite{Ra} is the
\begin{cor}
 \label{yngcor}
Let $A_j^{}$ be an operator of type j. Then
\begin{equation*}
A_j^{}:L^p(D)\rightarrow L^s(D) \quad
\frac{1}{s}>\frac{1}{p}-\frac{j}{2n+2}.
\end{equation*}
\end{cor}

We let $\lre_{1-2n}^i(\zeta,z)$ be a kernel of the form
\begin{equation*}
\lre_{1-2n}^i(\zeta,z)=
 \frac{\lre_{m,0}(\zeta,z)}{\rho^{2k}(\zeta,z)},
\end{equation*}
where $m-2k\ge 1-2n$. We denote by $E_{1-2n}$ the corresponding
isotropic operator.  The following theorem follows from
\cite{LiMi}.

\begin{thrm} \label{E1properties}
Then we have the following properties:
\begin{equation*}
E_{1-2n}:L^p(D)\rightarrow L^s(D)
\end{equation*}
for any $1\le p\le s\le\infty$ with $1/s>1/p-1/2n$.
\end{thrm}

\section{Basic integral representation}
 In this section we present the
basic integral representation for forms on bounded smooth strictly
pseudoconvex domains as worked out by Lieb and Range \cite{LR86}.

We start with the differential forms
\begin{align*}
&\beta(\zeta,z)=
 \frac{\partial_{\zeta}\rho^2(\zeta,z)}{\rho^2(\zeta,z)}\\
&\alpha_{\epsilon}(\zeta,z)=\xi(\zeta)\frac{\partial
r_{\epsilon}(\zeta)}{\phi_{\epsilon}(\zeta,z)},
\end{align*}
where $\xi(\zeta)$ is a smooth patching function which is
equivalently 1 for $|r(\zeta)|<\delta$ and 0 for
$|r(\zeta)|>\frac{3}{2} \delta$, and $\delta>0$ is sufficiently
small.  We define
\begin{equation*}
C_q^{\epsilon}=C_q(\alpha_{\epsilon},\beta)
 = \sum_{\mu=0}^{n-q-2}\sum_{\nu=0}^{q}
 a_{q\mu\nu}C_{q\mu\nu}(\alpha_{\epsilon},\beta),
\end{equation*}
where
\begin{equation*}
a_{q\mu\nu}=\left(\frac{1}{2\pi
i}\right)^n\binom{\mu+\nu}{\mu}\binom{n-2-\mu-\nu}{q-\mu}
\end{equation*}
and
\begin{equation*}
C_{q\mu\nu}(\alpha_{\epsilon},\beta)
 = \alpha_{\epsilon}\wedge
 \beta\wedge(\mdbar_{\zeta}\alpha_{\epsilon})^{\mu}\wedge
 (\mdbar_{\zeta}\beta)^{n-q-\mu-2}\wedge(\mdbar_z\alpha_{\epsilon})^{\nu}
 \wedge(\mdbar_z\beta)^{q-\nu}.
\end{equation*}
Denoting the Hodge $\ast$-operator by $\ast$, we then define
\begin{equation*}
\lrl_q^{\epsilon}(\zeta,z)=(-1)^{q+1}\ast_{\zeta}\overline{C_q^{\epsilon}(\zeta,z)}.
\end{equation*}
We also write
\begin{equation*}
K_q^{\epsilon}(\zeta,z) =
(-1)^{q(q-1)/2}\binom{n-1}{q}\frac{1}{(2\pi
i)^n}\alpha_{\epsilon}\wedge(\mdbar_{\zeta}\alpha_{\epsilon})^{n-q-1}\wedge(\mdbar_z\alpha_{\epsilon})^q
\end{equation*}
and
\begin{equation*}
\Gamma_{0,q}^{\epsilon}(\zeta,z)
 =\frac{(n-2)!}{2\pi^n}\frac{1}{\rho^{2n-2}}\left(
 \mdbar_{\zeta}\mdbar_z\rho^2\right)^q.
\end{equation*}

The kernels in our integral representation are defined through the
following for $q\ge1$:
\begin{align*}
\lrt^{a\epsilon }_q(\zeta,z)&=\vartheta_{\zeta}\lrl^{\epsilon}_q(\zeta,z)-\partial_z\lrl^{\epsilon}_{q-1}(\zeta,z), \\
\lrt^{i\epsilon }_q(\zeta,z)&=\mdbar_{\zeta}\Gamma_{0,q}^{\epsilon}(\zeta,z), \\
\lrt^{\epsilon}_q(\zeta,z)&=\lrt^{a\epsilon }_q(\zeta,z)+\lrt^{i\epsilon }_q(\zeta,z)\\
\lrp^{\epsilon}_q(\zeta,z)&=\lrq^{\epsilon}_q(\zeta,z)-\lrq^{\epsilon\ast}_q(\zeta,z) \\
&=\vartheta_{\zeta}\partial_z\lrl^{\epsilon}_{q-1}(\zeta,z)-(\vartheta_{\zeta}\partial_z\lrl^{\epsilon}_{q-1}(\zeta,z))^{\ast}\\
\lrq^{\epsilon}_q(\zeta,z)&
=\vartheta_{\zeta}\partial_z\lrl^{\epsilon}_{q-1}(\zeta,z).
\end{align*}
We denote the operators with kernels $\lrt_q^{\epsilon}$ and
$\lrp_q^{\epsilon}$
by ${\mathbf T_q^{\epsilon}}$ and ${\mathbf P_q^{\epsilon}}$,
respectively.

As mentioned above our goal is to establish $C^k$-estimates on the
Henkin-Leiterer domain, $D$, by exhausting $D$ by smooth strictly
pseudoconvex domains, $\{D_{\epsilon}\}_{\epsilon}$ and using the
analysis of Lieb and Range \cite{LR86} on the smooth domains
$D_{\epsilon}$. It is therefore necessary to be able to
approximate a given form
$f\in\mbox{Dom}(\mdbar^{\ast})\cap\mbox{Dom}(\mdbar)$ by forms
$f_{\epsilon}$ such that
\begin{align*}
&f_{\epsilon}\overset{L^2}{\rightarrow}f\\
&\mdbar f_{\epsilon}\overset{L^2}{\rightarrow}\mdbar f\\
&\mdbar^{\ast}_{\epsilon}f_{\epsilon}\overset{L^2}{\rightarrow}\mdbar
f.
\end{align*}
For this purpose we define the graph norm on $D$
\begin{equation*}
\|u\|_{G}^2=\|u\|^2+\| \mdbar u\|^2 +\|\mdbar^{\ast}u\|^2.
\end{equation*}

With
$H_{\epsilon}=\mdbar\mdbar^{\ast}_{\epsilon}+\mdbar^{\ast}_{\epsilon}\mdbar+I$,
\[
 \mbox{Dom}(H_{\epsilon})=\{f\in\mbox{Dom}(\mdbar^{\ast}_{\epsilon})
 \cap\mbox{Dom}(\mdbar)
 |\mdbar f\in\mbox{Dom}(\mdbar^{\ast}_{\epsilon}),
  \mdbar^{\ast}_{\epsilon} f\in\mbox{Dom}(\mdbar)\},
  \]
  and $\square_{\epsilon}$ defined by
\begin{equation*}
\square_{\epsilon}=\mdbar\mdbar^{\ast}_{\epsilon}+
 \mdbar^{\ast}_{\epsilon}\mdbar,
\end{equation*}
 we make the following
\begin{defn}
We say $f$ is in the space $\mathcal{M}_{(p,q)}(D)$,
$f\in\mathcal{M}_{(p,q)}(D)$, if $f$ is the limit in
$L^2_{(p,q)}(D;\mbox{loc})$ of $f_{\epsilon}\in \mbox{Dom
}H_{\epsilon}$  such that
$\sup_{\epsilon}\{\|f_{\epsilon}\|_{G,\epsilon},
 \|\square_{\epsilon}f_{\epsilon}\|_{\epsilon}\}<\infty$.
\end{defn}

 From \cite{HI} we have the following
\begin{prop}
\label{mdense}
 $\mathcal{M}_{(p,q)}(D)$ is dense in
$\mbox{Dom}(\mdbar^{\ast})\cap$Dom(\dbar) for the graph norm.
\end{prop}

Let $f\in L^2_{0,q}(D)\cap
\mbox{Dom}(\mdbar^{\ast})\cap\mbox{Dom}(\mdbar)$. We take a
sequence $\{f_{\epsilon}\}_{\epsilon}$ such that $f_{\epsilon}\in
\mbox{Dom }H_{\epsilon}$ and $f_{\epsilon}\rightarrow f$ in the
graph norm.

For each $f_{\epsilon}$ we apply the analysis of \cite{LR86} on
$D_{\epsilon}$, taking into account factors of $\gamma$, and
obtain the integral representation
\begin{thrm}
\label{bir}
\begin{align*}
f_{\epsilon}(z)=&{\mathbf T}^{\epsilon}_q\mdbar f_{\epsilon}
+({\mathbf
T}^{\epsilon}_{q-1})^{\ast}\mdbar_{\epsilon}^{\ast}f_{\epsilon}
+{\mathbf P}^{\epsilon}_qf_{\epsilon}\\
\nonumber
 &
 +\left(A_{(0,2)}^{\epsilon} + E_{2-2n}\right)\mdbar f_{\epsilon}
 +E_{2-2n}\mdbar_{\epsilon}^{\ast} f_{\epsilon}
 +\left(\frac{1}{\gamma^{\ast}}A_{(-1,1)}^{\epsilon}+E_{1-2n}\right)f_{\epsilon}.
\end{align*}
\end{thrm}
The proof follows as in \cite{LiMi}, but since the factors of
$\gamma$ are of particular importance here, we sketch the proof
including this new detail.
\begin{proof}[Sketch of proof.]
Our starting point is the Bochner-Martinelli-Koppelman (BMK)
formula for $f\in C^1_{0,q}(\overline{D}_{\epsilon})$.  Let $B_q$
be defined by
\begin{equation*}
B_q=\Omega_q(\beta)=(-1)^{q(q-1)/2}\binom{n-1}{q}\frac{1}{(2\pi
i)^n}
 \beta\wedge(\mdbar_{\zeta}\beta)^{n-q-1}
 \wedge(\mdbar_z\beta)^q.
\end{equation*}
Then for $z\in D_{\epsilon}$
\begin{multline}
\label{intrepwbndry}
 f(z)=\int_{\partial D_{\epsilon}}f(\zeta)\wedge B_q(\zeta,z)
  -\int_{D_{\epsilon}}\mdbar f(\zeta)\wedge B_q(\zeta,z)
  -\mdbar_z\int_{D_{\epsilon}}f(\zeta)\wedge B_{q-1}(\zeta,z)\\
  +(f(\zeta),\lre_{1-2n}(\zeta,z))+(\mdbar
  f(\zeta),\lre_{2-2n}(\zeta,z)).
 \end{multline}

Define the kernels $K_q^{\epsilon}(\zeta,z)$ by
$K_q^{\epsilon}(\zeta,z)=\Omega_q(\alpha_{\epsilon})$.  We then
proceed to replace the boundary integral in the BMK formula by
\begin{equation*}
\int_{\partial D_{\epsilon}}f\wedge K_q^{\epsilon}.
\end{equation*}
Let $\zeta_0\in\partial D_{\epsilon}$ be a fixed point and $U$ a
sufficiently small neighborhood of $\zeta_0$. $F(\zeta,z)$
vanishes on the diagonal of $\overline{U}\times\overline{U}$, so
Hefer's theorem applies to give us
\begin{equation*}
F(\zeta,z)=\sum_{j=1}^n h_j(\zeta,z)(\zeta_j-z_j).
\end{equation*}
We set
\begin{equation*}
 \alpha^0(\zeta,z)=\frac{\sum_{j=1}^n
h_j(\zeta,z)d\zeta_j}{F}.
\end{equation*}
With the metric given by
\begin{equation*}
ds^2=\sum g_{jk}(\zeta) dz_jd\overline{z}_k,
\end{equation*}
(recall the Levi metric is independent of $\epsilon$) we define
\begin{align*}
&b^0(\zeta,z)=\sum_{j,k=1}^n
g_{jk}(\overline{\zeta}_k-\overline{z}_k)d\zeta_j\\
 &R^2(\zeta,z)=
 \sum_{j,k=1}^n
 g_{jk}(\zeta_j-z_j)(\overline{\zeta}_k-\overline{z}_k)\\
&\beta^0(\zeta,z)=\frac{b^0(\zeta,z)}{R^2(\zeta,z)}.
\end{align*}
With use of the transition kernels $C_q$ defined above, we have
via Koppelman's homotopy formula
\begin{equation*}
\Omega_q(\beta^0)=\Omega_q(\alpha^0)
 +(-1)^{q+1}\mdbar_{\zeta}C_q(\alpha^0,\beta^0)
 +\mdbar_zC_{q-1}(\alpha^0,\beta^0).
\end{equation*}
On $(\partial D\cap U)\times U$ we have
\begin{align*}
&\Omega_q(\alpha^{\epsilon})=\Omega_q(\alpha_0)+\lre_{\infty}\\
&\Omega_0(\alpha^{\epsilon})=\Omega_0(\alpha_0)
 +\frac{\lre_1(\zeta,z)}{\phi_{\epsilon}(\zeta,z)^n}
\end{align*}
and on $U\times U$ we have
\begin{equation*}
\Omega_q(\beta)=\frac{R^{2n}}{\rho^{2n}}\Omega_q\beta^0
 +\lre_{2-2n}.
\end{equation*}
Thus we write
\begin{align*}
\Omega_q(\beta)&=\frac{R^{2n}}{\rho^{2n}}\Omega_q(\beta^0)
 +\lre_{2-2n}\\
 &=\frac{R^{2n}}{\rho^{2n}}(\Omega_q(\beta^0)-\Omega_q(\alpha^0))
 +\Omega_q(\alpha^0)
 +\frac{\lre_{2n+1}}{\rho^{2n}}\Omega_q(\alpha^0)+\lre_{2-2n},
\end{align*}
by which it then follows from the homotopy formula and the
relations between $b^0$ and $\rho^2$ and $R^2$, exactly as it was
obtained in \cite{LiMi}, that we have
\begin{align}
\label{trans}
 \Omega_q(\beta)
 =&\Omega_q(\alpha_{\epsilon})+
 (-1)^{q+1}\mdbar_{\zeta}C_q^{\epsilon}+\mdbar_zC_{q-1}^{\epsilon}+
(\Omega_q(\alpha^0)-\Omega_q(\alpha_{\epsilon})) +
\frac{\lre_{2n+1}}{\rho^{2n}}\Omega_q(\alpha^0)\\
\nonumber
 &+\lre_{2-2n}+\mdbar_{\zeta}\Bigg[\Bigg(
 C_q\left(\alpha^0,\frac{\partial\rho^2+\lre_2}{\rho^2}\right)
 -C_q(\alpha_{\epsilon},\beta)\Bigg)\\
 \nonumber
  &\qquad+\sum_{\mu,\nu}\frac{\lre_{3+2\mu+2\nu}}{(\rho^2)^{1+\mu+\nu}}
  C_{q\mu\nu}\left(\alpha^0,\frac{\partial\rho^2+\lre_2}{\rho^2}\right)
  \Bigg]\\
  \nonumber
  &+\sum_{\mu,\nu}\frac{\lre_{4+2\mu+2\nu}}{(\rho^2)^{2+\mu+\nu}}
  C_{q\mu\nu}\left(\alpha^0,\frac{\partial\rho^2+\lre_2}{\rho^2}\right)\\
  \nonumber
&+\sum_{\mu,\nu}\frac{\lre_{4+2\mu+2\nu}}{(\rho^2)^{2+\mu+\nu}}
  C_{(q-1)\mu\nu}\left(\alpha^0,\frac{\partial\rho^2+\lre_2}{\rho^2}\right)\\
  \nonumber
&+\mdbar_{z}\Bigg[\Bigg(
 C_{q-1}\left(\alpha^0,\frac{\partial\rho^2+\lre_2}{\rho^2}\right)
 -C_{q-1}(\alpha_{\epsilon},\beta)\Bigg)\\
 \nonumber
  &\qquad+\sum_{\mu,\nu}\frac{\lre_{3+2\mu+2\nu}}{(\rho^2)^{1+\mu+\nu}}
  C_{(q-1)\mu\nu}\left(\alpha^0,\frac{\partial\rho^2+\lre_2}{\rho^2}\right)
  \Bigg].
\end{align}
We now work with $\zeta\in\partial D_{\epsilon}$ so that
$F=\phi_{\epsilon}$.

For $q>0$ we have
$\Omega_q(\alpha_{\epsilon})=\Omega_q(\alpha^0)=0$ near the
boundary diagonal. Furthermore,
\begin{equation*}
C_{q\mu\nu}\left(\alpha^0,\frac{\partial\rho^2+\lre_2}{\rho^2}\right)
 =R_1\frac{\lre_1}{\phi_{\epsilon}^{1+\mu+\nu}(\rho^2)^{n-1-\mu-\nu}},
\end{equation*}
and thus
\begin{align*}
\mdbar_{\zeta}\Bigg(
 \sum_{\mu,\nu}&\frac{\lre_{3+2\mu+2\nu}}{(\rho^2)^{1+\mu+\nu}}
  C_{q\mu\nu}\left(\alpha^0,\frac{\partial\rho^2+\lre_2}{\rho^2}\right)
 \Bigg)=
 \mdbar_{\zeta}\left(
  R_1\frac{\lre_{4+2\mu+2\nu}}{\phi_{\epsilon}^{1+\mu+\nu}\rho^{2n}}
\right)\\
=&\frac{\lre_{4+2\mu+2\nu}}{\phi_{\epsilon}^{1+\mu+\nu}\rho^{2n}}
 +R_1\frac{\lre_{3+2\mu+2\nu}}{\phi_{\epsilon}^{1+\mu+\nu}\rho^{2n}}
+R_1\frac{\lre_{5+2\mu+2\nu}}{\phi_{\epsilon}^{1+\mu+\nu}(\rho^{2})^{n+1}}
\\
&
+R_1\frac{\lre_{5+2\mu+2\nu}+\lre_{4+2\mu+2\nu}\wedge\mdbar_{\zeta}r_{\epsilon}}{\phi_{\epsilon}^{2+\mu+\nu}\rho^{2n}}.
\end{align*}
And a similar formula holds for the
\begin{equation*}
\mdbar_z\left(\sum_{\mu,\nu}\frac{\lre_{3+2\mu+2\nu}}{(\rho^2)^{1+\mu+\nu}}
  C_{(q-1)\mu\nu}\left(\alpha^0,\frac{\partial\rho^2+\lre_2}{\rho^2}\right)\right)
\end{equation*}
term.

For $\nu>0$ we have
$C_{q\mu\nu}\left(\alpha^0,\frac{\partial\rho^2+\lre_2}{\rho^2}\right)
 =C_{q\mu\nu}(\alpha_{\epsilon},\beta)=0$ near
 the boundary diagonal, and for $\nu=0$ we have
\begin{equation}
\label{improvement}
C_{q\mu0}\left(\alpha^0,\frac{\partial\rho^2+\lre_2}{\rho^2}\right)
 -C_{q\mu0}(\alpha_{\epsilon},\beta)=
\frac{\lre_2}{\phi_{\epsilon}^{1+\mu}(\rho^2)^{n-1-\mu}},
\end{equation}
and thus
\begin{multline*}
\mdbar_{\zeta}\Bigg(C_{q\mu0}\left(\alpha_{\epsilon}^0,\frac{\partial\rho^2+\lre_2}{\rho^2}\right)
 -C_{q\mu0}(\alpha_{\epsilon},\beta)\Bigg)
  =\\
  \frac{\lre_1}{\phi_{\epsilon}^{1+\mu}(\rho^2)^{n-1-\mu}}+
 \frac{\lre_3+\lre_2\wedge\mdbar r_{\epsilon}}
{\phi_{\epsilon}^{2+\mu}(\rho^2)^{n-1-\mu}}+
\frac{\lre_3}{\phi_{\epsilon}^{1+\mu}(\rho^2)^{n-\mu}}.
\end{multline*}

An analogous formula holds for
\begin{equation*}
\mdbar_{z}\Bigg(C_{q\mu0}\left(\alpha^0,\frac{\partial\rho^2+\lre_2}{\rho^2}\right)
 -C_{q\mu0}(\alpha_{\epsilon},\beta)\Bigg)
\end{equation*}

(\ref{trans}) can thus be written
\begin{align*}
 \Omega_q(\beta)=&\Omega_q(\alpha_{\epsilon})
  +(-1)^{q+1}\mdbar_{\zeta}C_q(\alpha_{\epsilon},\beta)
  +\mdbar_zC_{q-1}(\alpha_{\epsilon},\beta)+\lre_{2-2n}\\
  &+\sum_{\tau=0}^n \frac{\lre_{3+2\tau}}{\phi_{\epsilon}^{1+\tau}\rho^{2n}}
 +\frac{\lre_{4+2\tau}\wedge\mdbar r_{\epsilon}}{\phi_{\epsilon}^{2+\tau}\rho^{2n}}
+R_1\frac{\lre_{5+2\tau}}{\phi_{\epsilon}^{1+\tau}(\rho^{2})^{n+1}}
 .
 \end{align*}

 Thus, after integrating by parts we obtain
\begin{align*}
\int_{\partial D_{\epsilon}}
 f\wedge &B_q^{\epsilon}  =\\
 & \int_{\partial D_{\epsilon}} f\wedge\Omega_q(\alpha_{\epsilon})
   + \int_{\partial D_{\epsilon}} \mdbar f\wedge
  C_q^{\epsilon} +\mdbar_z\int_{\partial D_{\epsilon}} f\wedge
  C_{q-1}^{\epsilon}+ \int_{\partial
  D_{\epsilon}}f\wedge\lre_{2-2n}\\
 &+ \sum_{\tau=0}^n\int_{\partial D_{\epsilon}}f\wedge
  \left( \frac{\lre_{3+2\tau}}{\phi_{\epsilon}^{1+\tau}\rho^{2n}}
+R_1\frac{\lre_{5+2\tau}}{\phi_{\epsilon}^{1+\tau}(\rho^{2})^{n+1}}
\right)
\end{align*}
We now replace all occurrences of $\rho^2$ in the denominators by
$P_{\epsilon}$, since the two are equal on $\partial
D_{\epsilon}$, and then we change the boundary integrals to volume
integrals by Stoke's Theorem:
\begin{align*}
\int_{\partial D_{\epsilon}}
 f\wedge &B_q^{\epsilon}  =\\
 & \int_{D_{\epsilon}} \mdbar f\wedge\Omega_q(\alpha_{\epsilon})
   +(-1)^q\int_{D_{\epsilon}} f\wedge
   \mdbar_{\zeta}\Omega_q(\alpha_{\epsilon})
   +(-1)^{q+1}\int_{D_{\epsilon}} \mdbar f\wedge
  \mdbar_{\zeta}C_q^{\epsilon}\\
   &+\int_{D_{\epsilon}} \mdbar f\wedge
  \mdbar_zC_{q-1}^{\epsilon}
+(-1)^q\mdbar_{z}\int_{D_{\epsilon}} f\wedge
  \mdbar_{\zeta}C_{q-1}^{\epsilon}
  + \int_{D_{\epsilon}}\mdbar f\wedge\lre_{2-2n}\\
 &
+ \int_{D_{\epsilon}}f\wedge\lre_{1-2n}
  + \sum_{\tau=0}^n\int_{D_{\epsilon}}\mdbar f\wedge
   \left( \frac{\lre_{3+2\tau}}{\phi_{\epsilon}^{1+\tau}P_{\epsilon}^{n}}
+R_1\frac{\lre_{5+2\tau}}{\phi_{\epsilon}^{1+\tau}P_{\epsilon}^{n+1}}
\right)\\
& + \sum_{\tau=0}^n\int_{D_{\epsilon}}f\wedge
  \Bigg(
\frac{\lre_{2+2\tau}}{\phi_{\epsilon}^{1+\tau}P_{\epsilon}^{n}}
 +R_1\frac{\lre_{3+2\tau}}{\phi_{\epsilon}^{2+\tau}P_{\epsilon}^{n}}
+\frac{\lre_{4+2\tau}}{\phi_{\epsilon}^{1+\tau}P_{\epsilon}^{n+1}}
+ \frac{r_{\epsilon}^{\ast}}{\gamma^{\ast}}
\frac{\lre_{3+2\tau}}{\phi_{\epsilon}^{1+\tau}P_{\epsilon}^{n+1}}\\
& \qquad
+R_2\frac{\lre_{5+2\tau}}{\phi_{\epsilon}^{2+\tau}P_{\epsilon}^{n+1}}
 +R_1\frac{\lre_{6+2\tau}}{\phi_{\epsilon}^{1+\tau}P_{\epsilon}^{n+2}}
+R_1\frac{r_{\epsilon}^{\ast}}{\gamma^{\ast}}
\frac{\lre_{5+2\tau}}{\phi_{\epsilon}^{1+\tau}P_{\epsilon}^{n+2}}
\Bigg).
\end{align*}

Inserting this expression of the boundary integral into
(\ref{intrepwbndry}), and using our notation of operators of a
certain type, we can write
\begin{align*}
f(z)=&(-1)^q\int_{D_{\epsilon}}f\wedge
 \mdbar_{\zeta}\Omega_q(\alpha_{\epsilon}) +\int_{D_{\epsilon}}
 \mdbar f\wedge \Omega_q(\alpha_{\epsilon})
 +(-1)^{q+1}\int_{D_{\epsilon}}\mdbar f\wedge
 \mdbar_{\zeta} C_q^{\epsilon}\\
 &-\int_{D_{\epsilon}}\mdbar f\wedge B_q^{\epsilon}
  +\mdbar_z\left( (-1)^q\int_{D_{\epsilon}}f\wedge \mdbar_{\zeta}
   C_{q-1}^{\epsilon}-\int_{D_{\epsilon}}f\wedge
   B_{q-1}^{\epsilon}\right)\\
   & +(f,\lre_{1-2n})
   +\left(f,\frac{1}{\gamma^{\ast}}\lra^{\epsilon}_{(-1,1)}\right)
  +(\mdbar f,\lre_{2-2n})
  +\left(\mdbar f,\lra_{(0,2)}^{\epsilon}\right)
\end{align*}

The rest of the proof follows as in \cite{LR86} to obtain a
rearrangement of the terms, and we arrive at the form of the
representation in our theorem.
\end{proof}

\section{Cancellation of singularities}

\begin{lemma}
\label{c1diff}
\begin{equation*}
\frac{r_{\epsilon}}{\gamma}\in C^1(D_{\epsilon})
\end{equation*}
with $C^1$-estimates independent of $\epsilon$.
\end{lemma}
\begin{proof}
Since $r_{\epsilon}\overset{C^1}\rightarrow r$, we show that
\begin{equation*}
\frac{r}{\gamma}\in C^1(D).
\end{equation*}

Outside of a neighborhood of any critical point of $r$, the result
is obvious.  We denote the critical points of $r$ by
$p_1,\ldots,p_k$, and take $\varepsilon$ small enough so that in
each
\begin{equation*}
U_{2\varepsilon}(p_j)=\{\zeta: D\cap |\zeta-p_j|<2\varepsilon\},
\end{equation*}
for $j=1,\ldots,k$,
 there are coordinates $u_{j_1},\ldots,u_{j_m},v_{j_{m+1}},\ldots,v_{j_{2n}}$ such
 that
 \begin{equation}
 \label{rcoor}
 -r(\zeta)=u_{j_1}^2+\cdots+u_{j_m}^2-v_{j_{m+1}}^2-\cdots - v_{j_{2n}}^2,
 \end{equation}
with $u_{j_{\alpha}}(p_j)=v_{j_{\beta}}(p_j)=0$ for all $1\le
\alpha\le m$ and $m+1\le\beta\le 2n$, from the Morse Lemma.

In these coordinates
\begin{equation*}
\frac{r(\zeta)}{\gamma(\zeta)}=-\frac{\ \
u_j^2-v_j^2}{\sqrt{u_j^2+v_j^2}},
\end{equation*}
where $u_j^2=u_{j_1}^2+\cdots+u_{j_m}^2$ and
$v_j^2=v_{j_{m+1}}^2+\cdots+ v_{j_{2n}}^2$.

It is then easy to see $r(\zeta)/\gamma(\zeta)$ is in $C^1$ by
differentiating with respect to the given coordinates.
\end{proof}
\label{cncl}
\begin{thrm}
\label{thrmT}
\begin{equation*}
{\mathbf T}^{\epsilon}_q = E_{1-2n}+A_{1}^{\epsilon}.
\end{equation*}
\end{thrm}
\begin{proof}
The proof is a direct result of the operators which make up
${\mathbf T}^{\epsilon}_q$.  In particular, the kernels
$\lrl_q^{\epsilon}$ and $\lrl_{q-1}^{\epsilon}$ are sums of terms
of the form
\begin{equation}
 \label{kerl}
R_1(\zeta)\frac{\lre_1}{\overline{\phi}_{\epsilon}^{\mu+1}P_{\epsilon}^{n-\mu-1}},
\quad \mu\ge0,
\end{equation}
and if $X$ is an arbitrary vector field in either $\zeta$ or $z$,
we use
\begin{equation}
 \label{xpe}
X P_{\epsilon}=\lre_{1,0}+
 \gamma\lre_{0,0}+\gamma^{\ast}\lre_{0,0}
\end{equation}
to calculate derivatives of the kernels (\ref{kerl}).  (\ref{xpe})
follows from Lemma \ref{c1diff} and from the property
\begin{equation*}
\frac{r}{\gamma}=\gamma\lre_{0,0},
\end{equation*}
which also holds by the proof of Lemma \ref{c1diff}.
\end{proof}

The proof of the next theorem will take up the bulk of this
section.  If one calculates the type of the operators associated
with the operator ${\mathbf P}^{\epsilon}_q$ as we did in Theorem
\ref{thrmT} just by looking at two vector fields operating on the
kernels $\lrl_{q-1}^{\epsilon}$, the conclusion would be that
${\mathbf P}^{\epsilon}_q$ is an operator of double type $(-1,0)$.
However, the combination of the two terms involved in ${\mathbf
P}^{\epsilon}_q$ cancels one order of singularity in the kernels
and thus leads to better mapping properties.  We shall prove the
\begin{thrm}
 \label{thrmP}
\begin{equation*}
 {\mathbf P}^{\epsilon}_q=
  \frac{1}{\gamma}A_{(-1,1)}^{\epsilon}+\frac{1}{\gamma^{\ast}}A_{(-1,1)}^{\epsilon}.
\end{equation*}
\end{thrm}

The following lemma follows as in the smooth case (see
\cite{LR86}).
\begin{lemma}
\label{phisymm}
\begin{equation*}
\phi_{\epsilon}-\phi^{\ast}_{\epsilon}=\lre_3.
\end{equation*}
\end{lemma}

For all $\epsilon$ sufficiently small, we work in coordinate patch
near a boundary point of $D$ and define orthogonal frame of
$(1,0)$-forms on a neighborhood $U\cap D_{\epsilon}$ with
$\omega^1_{\epsilon},\ldots,\omega^n_{\epsilon}$ where $\partial
r_{\epsilon}=\gamma\omega^n_{\epsilon}$ as the orthogonal frame,
and $L_1^{\epsilon}, \dots, L_n^{\epsilon}$ comprising the dual
frame. These operators refer to the variable $\zeta$.  When they
are to refer to the variable $z$, they will be denoted by
$\Theta^j_{\epsilon}$ and $\Lambda_j^{\epsilon}$, respectively.
\begin{prop}
\label{proplnp}
\begin{align*}
&i)\ \gamma \Lambda_n^{\epsilon} P_{\epsilon}
 = -2 \overline{\phi}_{\epsilon}
  +\frac{\gamma}{\gamma^{\ast}}\left(\lre_{0,0}P_{\epsilon}+\lre_{2,0}\right)+\lre_{2,0}\\
 &ii)\ \gamma^{\ast}L_n^{\epsilon} P_{\epsilon}=-2 \phi_{\epsilon}^{\ast}
  +\frac{\gamma^{\ast}}{\gamma}\left(\lre_{0,0}P_{\epsilon}+\lre_{2,0}\right)+\lre_{2,0}
\end{align*}
\end{prop}
\begin{proof}
We follow the proof of Proposition 2.18 in \cite{LiMi}.  We prove
$i)$ since $ii)$ is a consequence of $i)$.

We have
\begin{equation*}
\Lambda_n^{\epsilon} P_{\epsilon}=\Lambda_n \rho^2
 +2\frac{r_{\epsilon}}{\gamma}-
  \frac{1}{\gamma^{\ast}}\lre_{0,0}\frac{r_{\epsilon}r_{\epsilon}^{\ast}}{\gamma\gamma^{\ast}}
  .
\end{equation*}

We fix the point $\zeta$ and choose local coordinates
$z^{\epsilon}$ such that
\begin{equation*}
 dz_j^{\epsilon}(\zeta)=\Theta_j^{\epsilon}(\zeta).
\end{equation*}
Working in a neighborhood of a singularity in the boundary and
using the coordinates in (\ref{rcoor}), we see
 $\frac{\partial}{\partial z_n^{\epsilon}}$ is a combination of
 derivatives with coefficients of the form $\xi_0(\zeta)$, while
 $\Lambda_n$ is a combination of
 derivatives with coefficients of the form $\xi_0(z)$, where
 $\xi_0$ is defined in (\ref{defnxi}).
  We have $\Lambda_n -\frac{\partial}{\partial z_n}$ is a sum of
  terms of the form
\begin{equation*}
(\xi_0(z)-\xi_0(\zeta))\Lambda^{\epsilon} =
\lre_{1,-1}\Lambda^{\epsilon},
\end{equation*}
where $\Lambda^{\epsilon}$ denotes a first order differential
operator, and the last line follows from
\begin{align*}
\frac{1}{\gamma(\zeta)}-\frac{1}{\gamma(z)}&
 =\frac{\gamma(z)-\gamma(\zeta)}{\gamma(\zeta)\gamma(z)}\\
&=\frac{1}{\gamma(z)}\frac{\gamma^2(z)-\gamma^2(\zeta)}
 {\gamma(\zeta)(\gamma(\zeta)+\gamma(z))}\\
&=\frac{1}{\gamma(z)}\frac{\xi_1(\zeta)\lre_1}
 {\gamma(\zeta)(\gamma(\zeta)+\gamma(z))}\\
&=\frac{1}{\gamma(z)}\frac{\lre_{1,0}}
 {(\gamma(\zeta)+\gamma(z))}\\
&\lesssim \frac{1}{\gamma(z)}\frac{\lre_{1,0}}
 {\gamma(z)}\\
&=\lre_{1,-2}.
\end{align*}

 We note that $\Theta_j^{\epsilon}=\left.\Theta_j\right|_{D_{\epsilon}}$,
 and therefore we will suppress the $\epsilon$ superscript in the
 variable $z_n^{\epsilon}$ as well as in the differential operators
 denoted by $\Lambda$.  We have
\begin{equation*}
\rho^2=R^2+\lre_3
\end{equation*}
and
\begin{align*}
\Lambda_n \rho^2&=\frac{\partial}{\partial z_n}R^2 +\lre_{2,-1}^{\ast}\\
&=-2(\overline{\zeta}_n-\overline{z}_n)+\lre_{2,-1}^{\ast},
\end{align*}
where the last line follows from $g_{jk}=2\delta_{jk}$ due to the
orthogonality of the $\Theta_j$.

Finally, this gives
\begin{align}
\nonumber
 \Lambda_n
P_{\epsilon}&=-2(\overline{\zeta}_n-\overline{z}_n)
 +2\frac{r_{\epsilon}}{\gamma}-
  \frac{1}{\gamma^{\ast}}\lre_{0,0}\frac{r_{\epsilon}r_{\epsilon}^{\ast}}{\gamma\gamma^{\ast}}
  +\lre_{2,-1}^{\ast}\\
 \label{lnp}
  &=-2(\overline{\zeta}_n-\overline{z}_n)
 +2\frac{r_{\epsilon}}{\gamma}-
  \frac{1}{\gamma^{\ast}}
  (\lre_{0,0}P_{\epsilon}+\lre_{2,0})
  +\lre_{2,-1}^{\ast}
  .
\end{align}

We compare (\ref{lnp}) to $\overline{\phi}_{\epsilon}$ by
calculating the Levi polynomial, $F_{\epsilon}(\zeta,z)$ in the
above coordinates:
\begin{align}
\nonumber
 \overline{\phi}_{\epsilon}(\zeta,z)
 &=\overline{F}_{\epsilon}(\zeta,z) - r_{\epsilon}(\zeta) +\lre_{2}\\
\label{phicoor}
 &=\gamma(\zeta)(\overline{\zeta}_n-\overline{z}_n)-r_{\epsilon}(\zeta)+\lre_{2}.
\end{align}
$i)$ then easily follows.
\end{proof}

\begin{prop}
 \label{2p-l2}
\begin{equation}
 \label{eqn2p-l2}
\gamma\gamma^{\ast}\left( 2P_{\epsilon} -\sum_{j<n}
|L_j\rho^2_{\epsilon}|^2 \right)
 =4|\phi_{\epsilon}|^2+r_{\epsilon}(\zeta)\lre_{2}+\lre_{3,1}
  +\lre_{3,1}^{\ast}+\frac{\gamma^{\ast}}{\gamma}\lre_{4,0}.
\end{equation}
\end{prop}
\begin{proof}
We use coordinates as in the proof of Proposition \ref{proplnp}.
In particular, we write
\begin{equation*}
L_j=\frac{\partial}{\partial \zeta_j}+\lre_{1,-1}\Lambda.
\end{equation*}
Thus,
\begin{align*}
|L_j\rho^2|^2 &=\left| \frac{\partial}{\partial
\zeta_j}\rho^2_{\epsilon}
\right|^2 +\lre_{3,-1}+\lre_{4,-2}\\
 &=4|\zeta_j-z_j|^2+\lre_{3,-1}+\lre_{4,-2}.
\end{align*}
We can then write
\begin{equation*}
2P_{\epsilon} -\sum_{j<n} |L_j\rho^2|^2=
 4|\zeta_n-z_n|^2+4\frac{r_{\epsilon}r_{\epsilon}^{\ast}}{\gamma\gamma^{\ast}}
 +\lre_{3,-1}+\lre_{4,-2}.
\end{equation*}

Furthermore, from (\ref{phicoor}) we have
\begin{align*}
\phi_{\epsilon}\overline{\phi}_{\epsilon}
 &=(\gamma(\zeta_n-z_n)-r_{\epsilon}(\zeta)+\lre_2)\overline{\phi}_{\epsilon}\\
 &=\gamma(\zeta_n-z_n)[\gamma(\overline{\zeta}_n-\overline{z}_n)
  -r_{\epsilon}(\zeta)+\lre_2]-r_{\epsilon}(\zeta)\overline{\phi}_{\epsilon}+
  \lre_2\overline{\phi}_{\epsilon}\\
 &=\gamma\gamma^{\ast}|\zeta_n-z_n|^2 -
 r_{\epsilon}(\zeta)[\gamma(\zeta_n-z_n)+\overline{\phi}_{\epsilon}]
 +r_{\epsilon}(\zeta)\lre_2+\gamma\lre_{3,0},
\end{align*}
where we use $\gamma(\zeta)=\gamma(z)+\lre_{1,0}$ in the last
step.

From Lemma \ref{phisymm} we have
\begin{align*}
\gamma(\zeta_n-z_n)+\overline{\phi}_{\epsilon}
 &=\gamma(\zeta_n-z_n)+\overline{\phi}_{\epsilon}^{\ast}
 +\lre_3\\
&=\gamma(\zeta_n-z_n)+\gamma^{\ast}(z_n-\zeta_n)-r_{\epsilon}(z)+\lre_2\\
&=-r_{\epsilon}(z)+\lre_2,
\end{align*}
and so we can write
\begin{equation*}
\phi_{\epsilon}\overline{\phi}_{\epsilon} =
\gamma\gamma^{\ast}|\zeta_n-z_n|^2
+r_{\epsilon}r_{\epsilon}^{\ast}+r_{\epsilon}\lre_2+\gamma\lre_{3,0}.
\end{equation*}
(\ref{eqn2p-l2}) now easily follows.

\end{proof}

From
\begin{equation}
 \label{Aq}
A_{q\mu
0}^{\epsilon}=\frac{1}{\phi_{\epsilon}^{\mu+1}P_{\epsilon}^{n-\mu-1}}
 \partial r_{\epsilon}\wedge\partial\rho^2\wedge(\mdbar_{\zeta}\partial_{\zeta}r_{\epsilon})^{\mu}
 \wedge(\mdbar_{\zeta}\partial_{\zeta}\rho^2)^{n-q-\mu-2}
 \wedge(\mdbar_z\partial_{\zeta}\rho^2)^q+\lre_0.
\end{equation}
we have
\begin{equation*}
\lrl_q=\sum_{\mu=0}^{n-q-2}\left(
g_{q\mu}C_{q\mu}^{\epsilon}+\frac{R_1(\zeta)\lre_2}{\phi_{\epsilon}^{\mu+1}P_{\epsilon}^{n-\mu-1}}\right)
+\lre_0,
\end{equation*}
 where
 \begin{equation}
 \label{defnC}
C_{q\mu}^{\epsilon}=\sum_{|Q|=q\atop
j<n}\frac{\overline{L}_j\rho^2}{\overline{\phi}_{\epsilon}^{\mu+1}
P_{\epsilon}^{n-\mu-1}}\overline{\omega}^{njQ}\wedge\Theta^Q
\end{equation}
and $g_{q\mu}(\zeta)$ is a real valued function of the form
$R_1(\zeta) \sigma_{q\mu}(\zeta)$ for a real valued function
$\sigma_{q\mu}$.  It follows that
\begin{equation*}
\lrq_{q+1}^{\epsilon}=\sum_{\mu=0}^{n-q-2}[g_{q\mu}(\zeta)\vartheta_{\zeta}\partial_z
C_{q\mu}^{\epsilon}-g_{q\mu}(z)(\vartheta_{\zeta}\partial_z
C_{q\mu}^{\epsilon})^{\ast}]
 +\frac{1}{\gamma^{\ast}}\lra_{1}^{\epsilon}.
\end{equation*}

\begin{prop}
\label{cnclprop}
 Let $C_{q\mu}^{\epsilon}$ be given by (\ref{defnC}).  Then
 \begin{equation*}
g_{q\mu}(\zeta)\vartheta_{\zeta}\partial_z
C_{q\mu}^{\epsilon}-g_{q\mu}(z)(\vartheta_{\zeta}\partial_z
C_{q\mu}^{\epsilon})^{\ast}=\frac{1}{\gamma}\lra_{(-1,1)}^{\epsilon}
 +\frac{1}{\gamma^{\ast}}\lra_{(-1,1)}^{\epsilon}.
 \end{equation*}
\end{prop}

We write
\begin{equation*}
\vartheta_{\zeta}\partial_z C_{q\mu}^{\epsilon}=\sum_{|K|=q+1\atop
|L|=q+1} A_{KL}^{\epsilon\mu} \overline{\omega}^K\wedge \Theta^L
\end{equation*}
and
\begin{equation*}
(\vartheta_{\zeta}\partial_z
C_{q\mu}^{\epsilon})^{\ast}=\sum_{|K|=q+1\atop |L|=q+1}
A_{LK}^{\epsilon\mu\ \ast} \overline{\omega}^K\wedge \Theta^L.
\end{equation*}
To prove Proposition \ref{cnclprop}, we show
\begin{equation*}
\gamma(\zeta)A_{KL}^{\epsilon\mu}-\gamma(z)A_{LK}^{\epsilon\mu\
\ast}=\frac{1}{\gamma}\lra_{(-1,1)}^{\epsilon}
 +\frac{1}{\gamma^{\ast}}\lra_{(-1,1)}^{\epsilon}.
\end{equation*}

With
\begin{equation*}
\lrm_{kj}^{\epsilon\mu}=\frac{1}{\overline{\phi}_{\epsilon}^{\mu+1}}
 \Lambda_k\left(\frac{\overline{L}_j\rho^2}{P_{\epsilon}^{n-\mu-1}}\right),
\end{equation*}
and
 using $\Lambda_k\overline{\phi}_{\epsilon}=\lre_{1,0}$, we have
\begin{equation*}
A_{KL}^{\epsilon\mu}=-\sum_{|Q|=q \atop {j<n \atop k,m}}
 \varepsilon_{kQ}^L\varepsilon_{mK}^{njQ}L_m\lrm_{kj}^{\epsilon\mu}+\lra_{(-1,1)}^{\epsilon}
 +\frac{1}{\gamma}\lra_{(0,2)}^{\epsilon},
\end{equation*}
where $K$, $L$, and $Q$ are multi-indices, and the symbol
$\varepsilon_{kQ}^L$ is defined by
\begin{equation*}
\varepsilon_{kQ}^L=
 \begin{cases}
 0 & \mbox{if } \{k\}\cup Q\neq L\\
 1 & \mbox{if } kQ  \mbox{ differs from } L \mbox{ by an even permutation} \\
 -1 & \mbox{if } kQ  \mbox{ differs from } L \mbox{ by an odd permutation.}
 \end{cases}
\end{equation*}

\begin{lemma}
\label{mlemma}
\begin{align*}
i)\ \lrm_{kj}^{\epsilon\mu} =&
\frac{1}{\overline{\phi}_{\epsilon}^{\mu+1}}
 \left[\frac{-2\delta_{kj}}{P_{\epsilon}^{n-\mu-1}}+\frac{n-\mu-1}{P_{\epsilon}^{n-\mu}}
 (L_k\rho^2)(\overline{L}_j\rho^2)\right]
 +error \quad k<n\\
ii)\ \lrm_{nj}^{\epsilon\mu} =&\frac{1}{\gamma(\zeta)}
\frac{2(n-\mu-1)
\overline{L}_j\rho^2}{\overline{\phi}_{\epsilon}^{\mu}
 P_{\epsilon}^{n-\mu}}+ error
\quad j<n,
\end{align*}
where "$error$" refers to error terms with the property that any
derivative with respect to the $\zeta$ variable leads to kernels
of the form
\begin{equation*}
\frac{1}{\gamma}\lra_{(-1,1)}^{\epsilon}+
 \frac{1}{\gamma^{\ast}}\lra_{(-1,1)}^{\epsilon}
 +\frac{1}{\gamma^2}\lra_{(0,2)}^{\epsilon}.
\end{equation*}
\end{lemma}
\begin{proof}
We will prove $ii)$, the proof of $i)$ following similar
arguments, and being easier to prove.

It is straightforward to check with the aid of coordinates chosen
as in Proposition \ref{2p-l2} that
\begin{equation*}
\overline{L}_j\rho^2= (\zeta_j-z_j)+\lre_{2,-1}.
\end{equation*}
When it is not necessary to refer to the special coordinates in
Proposition \ref{2p-l2}, we can also write
$\overline{L}_j\rho^2=\lre_1$.  We will also refer to the
calculation
\begin{align*}
\Lambda_n\overline{L}_j\rho^2=&\overline{L}_j\Lambda_n\rho^2\\
=&\overline{L}_j\left[\left(\frac{\partial}{\partial
z_n}+\lre_{1,-1}^{\ast}\Lambda\right)R^2\right]+\lre_{2,0}\\
=&\overline{L}_j( -(\overline{\zeta_n-z_n})+\lre_{2,-1}^{\ast})+\lre_{2,0}\\
=&-\left(\frac{\partial}{\partial
\overline{\zeta}_j}+\lre_{1,-1}\Lambda\right)(\overline{\zeta_n-z_n})+
\lre_{1,-1}^{\ast}\\
=&\lre_{1,-1}+\lre_{1,-1}^{\ast},
\end{align*}
for $j<n$, below.

 We have
\begin{align}
\nonumber
\lrm_{nj}^{\epsilon\mu}&=\frac{1}{\overline{\phi}_{\epsilon}^{\mu+1}}
 \Lambda_n\left(\frac{\overline{L}_j\rho^2}{P_{\epsilon}^{n-\mu-1}}\right)\\
 \nonumber
 =&\frac{\lre_{1,-1}+\lre_{1,-1}^{\ast}}{\overline{\phi}_{\epsilon}^{\mu+1}P_{\epsilon}^{n-\mu-1}}
 -(n-\mu-1)\frac{\overline{L}_j\rho^2}{\overline{\phi}_{\epsilon}^{\mu+1}P_{\epsilon}^{n-\mu}}
 \Lambda_nP_{\epsilon}\\
 \nonumber
=&\frac{\lre_{1,-1}+\lre_{1,-1}^{\ast}}{\overline{\phi}_{\epsilon}^{\mu+1}P_{\epsilon}^{n-\mu-1}}
-
\frac{1}{\gamma}\frac{(n-\mu-1)\overline{L}_j\rho^2}{\overline{\phi}_{\epsilon}^{\mu+1}P_{\epsilon}^{n-\mu}}
\left(-2 \overline{\phi}_{\epsilon}
  +\frac{\gamma}{\gamma^{\ast}}\left(\lre_{0,0}P_{\epsilon}+\lre_{2,0}\right)+\lre_2\right)\\
= \label{forerror}
 &
2\frac{1}{\gamma}\frac{(n-\mu-1)\overline{L}_j\rho^2}{\overline{\phi}_{\epsilon}^{\mu}P_{\epsilon}^{n-\mu}}
+
\frac{\lre_{1,-1}+\lre_{1,-1}^{\ast}}{\overline{\phi}_{\epsilon}^{\mu+1}P_{\epsilon}^{n-\mu-1}}
+ \frac{1}{\gamma^{\ast}}\frac{\lre_{3,0}}
{\overline{\phi}_{\epsilon}^{\mu+1}P_{\epsilon}^{n-\mu}}+
\frac{1}{\gamma}\frac{\lre_{3,0}}
{\overline{\phi}_{\epsilon}^{\mu+1}P_{\epsilon}^{n-\mu}},
\end{align}
from which $ii)$ easily follows.  We used Proposition
\ref{proplnp} in the third equality above.

Applying a differential operator with respect to $\zeta$ to the
last three terms in (\ref{forerror}) and noting the types of
kernels which arise finishes the proof.
\end{proof}

We start with the case $\mu=0$ and compute
 $L_m\lrm^{\epsilon 0}_{kj}$.
\begin{lemma}
 \label{derM}
Let $m,k,j<n$. Then
\begin{align*}
i)\ L_m
\lrm_{kj}^{\epsilon0}=&\frac{2(n-1)}{\overline{\phi}_{\epsilon}P^n_{\epsilon}}
 (\delta_{kj}L_m\rho^2 +\delta_{mj}L_k\rho^2) -
 \frac{n(n-1)}{\overline{\phi}_{\epsilon}P^{n+1}_{\epsilon}}
 (L_m\rho^2)(L_k\rho^2)(\overline{L}_j\rho^2)\\
 &+\frac{1}{\gamma}\lra_{(-1,1)}^{\epsilon}
  +\frac{1}{\gamma^{\ast}}\lra_{(-1,1)}^{\epsilon}
 +\frac{1}{\gamma^2}\lra_{(0,2)}^{\epsilon}\\
ii)\ L_m\lrm_{nj}^{\epsilon0}=& \frac{2(n-1)}{\gamma}
 \left(\frac{2\delta_{mj}}{P^n}-\frac{n}{P^{n+1}}(L_m\rho^2)
 (\overline{L}_j\rho^2)\right)+\frac{1}{\gamma^2}\lra_{(-1,1)}^{\epsilon}
 +\frac{1}{\gamma^{\ast}}\lra_{(-1,1)}^{\epsilon}
 \\
iii)\ L_n\lrm_{kj}^{\epsilon0}=&
 -\gamma\frac{2\delta_{kj}}{\overline{\phi}_{\epsilon}^2P^{n-1}_{\epsilon}}
  +\frac{(n-1)(L_k\rho^2)(\overline{L}_j\rho^2)}
  {\overline{\phi}_{\epsilon}^2P^n_{\epsilon}}\\
  & +\frac{1}{\gamma^{\ast}}
\left(\frac{2n(n-1)(L_k\rho^2)(\overline{L}_j\rho^2)}{P^{n+1}}
-\frac{4(n-1)\delta_{kj}}{P^n}\right)\frac{\phi_{\epsilon}^{\ast}}
{\overline{\phi}_{\epsilon}}\\
&+\frac{1}{\gamma}\lra_{(-1,1)}^{\epsilon}+\frac{1}{\gamma^{\ast}}\lra_{(-1,1)}^{\epsilon}
 +\frac{1}{\gamma^2}\lra_{(0,2)}^{\epsilon}
\\
iv)\ L_n\lrm_{nj}^{\epsilon0}
=&\frac{4n(n-1)}{\gamma\gamma^{\ast}}\frac{\phi_{\epsilon}^{\ast}}{P_{\epsilon}^{n+1}}
 \overline{L}_j\rho^2+\frac{1}{\gamma^2}\lra_{(-1,1)}^{\epsilon}
 +\frac{1}{\gamma\gamma^{\ast}}\lra_{(-1,1)}^{\epsilon}.
\end{align*}
\end{lemma}
\begin{proof}
$i)$.  We use
\begin{align}
\label{2ndpartmk}
 L_m\frac{1}{\overline{\phi}_{\epsilon}
 P_{\epsilon}^{n-1}} &=-\frac{1}{\overline{\phi}_{\epsilon}^2
 P_{\epsilon}^{n-1}}L_m\overline{\phi}_{\epsilon}
  -(n-1)\frac{1}{\overline{\phi}_{\epsilon}
 P_{\epsilon}^{n}}L_mP_{\epsilon}\\
 \nonumber
 &=-\frac{(n-1)}{\overline{\phi}_{\epsilon}
 P_{\epsilon}^{n}}L_m\rho^2+\frac{1}{\gamma}\lra_{(-1,1)}^{\epsilon},
\end{align}
since for $m<n$ we have
 $\L_mP_{\epsilon}=\lre_{0,0}(P_{\epsilon}+\lre_2)$,
 and
\begin{align}
\label{lmrt} L_m
\left(\frac{1}{\overline{\phi}_{\epsilon}P_{\epsilon}^{n}}
 (L_k\rho^2)(\overline{L}_j\rho^2)\right)=&
\frac{1}{\overline{\phi}_{\epsilon}P_{\epsilon}^{n}}
 (L_mL_k\rho^2)(\overline{L}_j\rho^2)+
\frac{1}{\overline{\phi}_{\epsilon}P_{\epsilon}^{n}}
 (L_k\rho^2)(L_m\overline{L}_j\rho^2)\\
 \nonumber
&+\frac{\lre_3}{\overline{\phi}_{\epsilon}^2P_{\epsilon}^{n}}-
\frac{n}{\overline{\phi}_{\epsilon}P_{\epsilon}^{n+1}}(L_k\rho^2)(\overline{L}_j\rho^2)(L_mP_{\epsilon}).
\end{align}
An easy calculation gives
\begin{align*}
L_mL_k\rho^2&=L_m\left(\overline{\zeta_k-z_k}+\lre_{2,-1}\right)\\
 &=\left(\frac{\partial}{\partial
 \zeta_m}+\lre_{1,-1}\Lambda\right)
 (\overline{\zeta_k-z_k})+L_m(\lre_{2,-1})\\
 &=\lre_{1,-1}+\lre_{2,-2}\\
L_m\overline{L}_j\rho^2&= 2\delta_{mj}+\lre_{1,-1}\\
L_mP_{\epsilon}&=L_m\rho^2+\frac{\lre_{0,0}}{\gamma}(P_{\epsilon}+\lre_2)
\end{align*}
so that the right hand side of (\ref{lmrt}) becomes
\begin{align}
\label{lastpartmk}
\frac{2\delta_{mj}}{\overline{\phi}_{\epsilon}P_{\epsilon}^{n}}
 (L_k\rho^2)
- &\frac{n}{\overline{\phi}_{\epsilon}P_{\epsilon}^{n+1}}
(L_k\rho^2)(\overline{L}_j\rho^2)(L_m\rho^2)
 +\frac{\lre_{2,-1}+\lre_{3,-2}}{\overline{\phi}_{\epsilon}P_{\epsilon}^{n}}
 +\frac{\lre_{4,-1}}{\overline{\phi}_{\epsilon}P_{\epsilon}^{n+1}}
 +\frac{\lre_3}{\overline{\phi}_{\epsilon}^2P_{\epsilon}^{n}}\\
 \nonumber
 =&\frac{2\delta_{mj}}{\overline{\phi}_{\epsilon}P_{\epsilon}^{n}}
 (L_k\rho^2)
- \frac{n}{\overline{\phi}_{\epsilon}P_{\epsilon}^{n+1}}
(L_k\rho^2)(\overline{L}_j\rho^2)(L_m\rho^2)
+\frac{1}{\gamma}\lra_{(-1,1)}^{\epsilon}+
\frac{1}{\gamma^2}\lra_{(0,2)}^{\epsilon}.
 \end{align}
(\ref{2ndpartmk}), and (\ref{lastpartmk}) together with the form
of $\lrm_{kj}^{\epsilon0}$ from Lemma \ref{mlemma} prove part
$i)$.

$ii)$.  We have
\begin{align*}
L_m\left(\frac{1}{\gamma(\zeta)} \frac{2(n-1)
\overline{L}_j\rho^2}{
 P_{\epsilon}^{n}}\right)=&
\frac{2(n-1)}{\gamma} \frac{ L_m\overline{L}_j\rho^2}{
 P_{\epsilon}^{n}}-\frac{2n(n-1)}{\gamma} \frac{ \overline{L}_j\rho^2}{
 P_{\epsilon}^{n+1}}L_mP_{\epsilon}
  +\frac{1}{\gamma^2}\frac{\lre_{1}}{P^n_{\epsilon}}
 \\
 =&\frac{4(n-1)}{\gamma} \frac{ \delta_{mj}}{
 P_{\epsilon}^{n}}-\frac{2n(n-1)}{\gamma} \frac{ \overline{L}_j\rho^2}{
 P_{\epsilon}^{n+1}}L_m\rho^2
\\
 &+\frac{1}{\gamma^2} \frac{\lre_{1}}{
 P_{\epsilon}^{n}}+\frac{1}{\gamma^2}\frac{\lre_{3,0}}{P_{\epsilon}^{n+1}}
\end{align*}
as in the proof of $i)$.  Taking into consideration the error
terms from Lemma \ref{mlemma}, we conclude $ii)$.

$iii)$.
\begin{align*}
L_n\left(
 \frac{-2\delta_{kj}}{\overline{\phi}_{\epsilon}P_{\epsilon}^{n-1}}\right)=&
\frac{2\delta_{kj}}{\overline{\phi}_{\epsilon}^2P_{\epsilon}^{n-1}}L_n\overline{\phi}_{\epsilon}
+2(n-1)\frac{\delta_{kj}}{\overline{\phi}_{\epsilon}P_{\epsilon}^{n}}L_nP_{\epsilon}\\
=&\frac{2\delta_{kj}}{\overline{\phi}_{\epsilon}^2P_{\epsilon}^{n-1}}(-\gamma(\zeta)+\lre_1)
\\
&+2(n-1)\frac{\delta_{kj}}{\overline{\phi}_{\epsilon}P_{\epsilon}^{n}}
\left( -2 \frac{\phi_{\epsilon}^{\ast}}{\gamma^{\ast}}
  +\frac{1}{\gamma}\left(\lre_{0,0}P_{\epsilon}+\lre_{2,0}\right)+\frac{\lre_{2,0}}{\gamma^{\ast}}\right)\\
=&-\gamma\frac{2\delta_{kj}}{\overline{\phi}_{\epsilon}^2P_{\epsilon}^{n-1}}
-\frac{1}{\gamma^{\ast}}\frac{4(n-1)\delta_{kj}}{P_{\epsilon}^{n}}
 \frac{\phi_{\epsilon}^{\ast}}{\overline{\phi}_{\epsilon}}
+\frac{1}{\gamma}\lra_{(-1,1)}^{\epsilon}+
\frac{1}{\gamma^{\ast}}\lra_{(-1,1)}^{\epsilon}.
\end{align*}
Since
\begin{align*}
L_nL_k\rho^2=&L_kL_n\rho^2+[L_n,L_k]\rho^2\\
=&L_k\left[\left(\frac{\partial}{\partial
\zeta_n}+\lre_{1,-1}\Lambda\right)R^2\right]+\lre_{1,-1}\\
=&L_k( \overline{\zeta_n-z_n}+\lre_{2,-1})+\lre_{1,-1}\\
=&\left(\frac{\partial}{\partial
\zeta_k}+\lre_{1,-1}\Lambda\right)(\overline{\zeta_n-z_n})+
\lre_{1,-1}+\lre_{2,-2}\\
=&\lre_{1,-1}+\lre_{2-2},
\end{align*}
we have
\begin{align*}
L_n\Bigg(&
 \frac{n-1}{\overline{\phi}_{\epsilon}P_{\epsilon}^{n}}(L_k\rho^2)(\overline{L}_j\rho^2)
 \Bigg)\\
 =&-\frac{n-1}{\overline{\phi}_{\epsilon}^2P_{\epsilon}^{n}}(L_n\overline{\phi}_{\epsilon})
 (L_k\rho^2)(\overline{L}_j\rho^2)
  -\frac{n(n-1)}{\overline{\phi}_{\epsilon}P_{\epsilon}^{n+1}}(L_nP_{\epsilon})
 (L_k\rho^2)(\overline{L}_j\rho^2)\\
 &+ \frac{n-1}{\overline{\phi}_{\epsilon}P_{\epsilon}^{n}}(L_nL_k\rho^2)(\overline{L}_j\rho^2)
 \\
 =&\gamma\frac{n-1}{\overline{\phi}_{\epsilon}^2P_{\epsilon}^{n}}
 (L_k\rho^2)(\overline{L}_j\rho^2)+\frac{\lre_3}{\overline{\phi}_{\epsilon}^2P_{\epsilon}^{n}}\\
&-\frac{n(n-1)}{\overline{\phi}_{\epsilon}P_{\epsilon}^{n+1}}
 (L_k\rho^2)(\overline{L}_j\rho^2)\left( -2 \frac{\phi_{\epsilon}^{\ast}}{\gamma^{\ast}}
  +\frac{1}{\gamma}\left(\lre_0P_{\epsilon}+\lre_2\right)+\frac{\lre_{2,0}}{\gamma^{\ast}}\right)
  +\frac{\lre_{2,-1}+\lre_{3,-2}}{\overline{\phi}_{\epsilon}P_{\epsilon}^{n}}\\
=&\gamma\frac{(n-1)(L_k\rho^2)(\overline{L}_j\rho^2)}{\overline{\phi}_{\epsilon}^2P_{\epsilon}^{n}}
 +\frac{1}{\gamma^{\ast}}\frac{2n(n-1)(L_k\rho^2)(\overline{L}_j\rho^2)}{P_{\epsilon}^{n+1}}
 \frac{\phi_{\epsilon}^{\ast}}{\overline{\phi}_{\epsilon}}\\
 &+\frac{1}{\gamma}\lra_{(-1,1)}^{\epsilon}+\frac{1}{\gamma^{\ast}}\lra_{(-1,1)}^{\epsilon}
  +\frac{1}{\gamma^2}\lra_{(0,2)}^{\epsilon}.
\end{align*}
Putting these calculations together and including the error terms
from Lemma \ref{mlemma}, we can prove $iii)$.

 $iv)$.  To prove $iv)$ we calculate
\begin{align*}
L_n\left(\frac{1}{\gamma(\zeta)} \frac{2(n-1)
\overline{L}_j\rho^2}{
 P_{\epsilon}^{n}}\right)
  =&-\frac{1}{\gamma^2} \frac{2(n-1)
\overline{L}_j\rho^2}{
 P_{\epsilon}^{n}}L_n\gamma+\frac{1}{\gamma} \frac{2(n-1)
L_n\overline{L}_j\rho^2}{
 P_{\epsilon}^{n}}\\
 &-\frac{1}{\gamma} \frac{2(n-1)
\overline{L}_j\rho^2}{
 P_{\epsilon}^{n+1}}L_nP_{\epsilon}\\
 =&
-\frac{1}{\gamma} \frac{2(n-1) \overline{L}_j\rho^2}{
 P_{\epsilon}^{n+1}}\left(
 -2 \frac{\phi_{\epsilon}^{\ast}}{\gamma^{\ast}}
  +\frac{1}{\gamma}\left(\lre_{0,0}P_{\epsilon}+\lre_{2,0}\right)
  +\frac{\lre_{2,0}}{\gamma^{\ast}}\right)\\
  &+\frac{1}{\gamma^2} \frac{\lre_1}{
 P_{\epsilon}^{n}}\\
 =& \frac{4n(n-1)}{\gamma\gamma^{\ast}}\frac{\phi_{\epsilon}^{\ast}}{P_{\epsilon}^{n+1}}
 \overline{L}_j\rho^2+\frac{1}{\gamma^2}\lra_{(-1,1)}^{\epsilon}
 +\frac{1}{\gamma\gamma^{\ast}}\lra_{(-1,1)}^{\epsilon}.
\end{align*}
The error terms may also be absorbed into the terms of the last
calculation.
\end{proof}

\begin{proof}[Proof of Proposition \ref{cnclprop}.]
To compute $A_{KL}^{\epsilon0}$ we follow \cite{LiMi} and consider
four cases:
 \newline
  {\it Case 1.} $n\in K$ and $n\in L$.
\begin{align*}
A_{KL}^{\epsilon0}&=\sum_{j,m<n}\varepsilon_{mK}^{jL}L_m\lrm_{nj}^{\epsilon0} \\
 A_{KL}^{\epsilon0\ \ast}&=\sum_{j,m<n}\varepsilon_{mK}^{jL}(L_j\lrm_{nm}^{\epsilon0})^{\ast}.
\end{align*}
By Lemma \ref{derM} ii),
\begin{equation*}
\gamma(\zeta)L_m\lrm_{nj}^{\epsilon0}-\gamma(z)(L_j\lrm_{nm}^{\epsilon0})^{\ast}=
 \frac{1}{\gamma}\lra_{(-1,1)}^{\epsilon}+\frac{1}{\gamma^{\ast}}\lra_{(-1,1)}^{\epsilon}.
\end{equation*}
\newline
{\it Case 2.} $n\notin K$ and $n\notin L$.  From \cite{LR86} (see
also \cite{LiMi})
\begin{align*}
A_{KL}^{\epsilon0}&= -\sum_{j\in K\atop k\in L} \varepsilon_{kQ}^L
\varepsilon_K^{jQ}L_n\lrm_{kj}^{\epsilon0}\\
A_{LK}^{\epsilon0\ \ast}&= -\sum_{j\in K\atop k\in L}
\varepsilon_{jQ}^K
\varepsilon_L^{kQ}(L_n\lrm_{jk}^{\epsilon0})^{\ast}.
\end{align*}
We refer to Lemma \ref{derM} iii) to calculate
\begin{equation*}
\gamma(\zeta)L_n\lrm_{kj}^{\epsilon0}-\gamma(z)(L_n\lrm_{jk}^{\epsilon0})^{\ast}.
\end{equation*}
We have
\begin{align}
\label{case21}
\frac{\gamma^2}{\overline{\phi}_{\epsilon}^2P_{\epsilon}^{n-1}}
 -\frac{(\gamma^{\ast})^2}{(\overline{\phi}_{\epsilon}^{\ast})^2P_{\epsilon}^{n-1}}
 &=\frac{\gamma^2}{P_{\epsilon}^{n-1}}
 \left(\frac{1}{\overline{\phi}_{\epsilon}^2}-\frac{1}{(\overline{\phi}_{\epsilon}^{\ast})^2}\right)
 +\frac{\lre_{1,1}}{\overline{\phi}_{\epsilon}^2P_{\epsilon}^{n-1}}\\
 \nonumber
 &=\frac{\gamma^2}{P_{\epsilon}^{n-1}}\frac{(\overline{\phi}_{\epsilon}+\overline{\phi}_{\epsilon}^{\ast})\lre_3}
 {\overline{\phi}_{\epsilon}^2(\overline{\phi}_{\epsilon}^{\ast})^2}+\lra_{(0,1)}^{\epsilon}\\
 \nonumber
 &=\lra_{(0,1)}^{\epsilon}.
\end{align}
Furthermore,
\begin{align}
\label{case22}
\gamma(\zeta)\frac{(L_k\rho^2)(\overline{L}_j\rho^2)}
  {\overline{\phi}_{\epsilon}^2P^n_{\epsilon}}-&\gamma(z)\left(
   \frac{(L_j\rho^2)(\overline{L}_k\rho^2)}
  {\overline{\phi}_{\epsilon}^2P^n_{\epsilon}}\right)^{\ast}\\
  \nonumber
  &=
  \gamma\left(\frac{(L_k\rho^2)(\overline{L}_j\rho^2)}
  {\overline{\phi}_{\epsilon}^2P^n_{\epsilon}}-\left(
   \frac{(L_j\rho^2)(\overline{L}_k\rho^2)}
  {\overline{\phi}_{\epsilon}^2P^n_{\epsilon}}\right)^{\ast}\right)
  +\lra_{(-1,1)}^{\epsilon}\\
\nonumber
  &=\frac{\lre_{2,1}}{P_{\epsilon}^n}
  \left(\frac{1}{\overline{\phi}_{\epsilon}^2}-\frac{1}{(\overline{\phi}_{\epsilon}^{\ast})^2}\right)
  +\lra_{(-1,1)}^{\epsilon}\\
\nonumber
  &=\lra_{(0,1)}^{\epsilon}+\lra_{(-1,1)}^{\epsilon}\\
  \nonumber
  &=\lra_{(1,-1)}^{\epsilon},
\end{align}
and
\begin{align}
\label{case23}
 \frac{\gamma}{\gamma^{\ast}}
\frac{(L_k\rho^2)(\overline{L}_j\rho^2)}{P^{n+1}_{\epsilon}}\frac{\phi_{\epsilon}^{\ast}}{\overline{\phi}_{\epsilon}}
 &-\frac{\gamma^{\ast}}{\gamma}
\left(\frac{(L_j\rho^2)(\overline{L}_k\rho^2)}{P^{n+1}_{\epsilon}}\frac{\phi_{\epsilon}^{\ast}}{\overline{\phi}_{\epsilon}}\right)^{\ast}\\
\nonumber
 &=\frac{\lre_2}{P_{\epsilon}^{n+1}}\left(
 \frac{\gamma}{\gamma^{\ast}}\frac{\phi_{\epsilon}^{\ast}}{\overline{\phi}_{\epsilon}}-
 \frac{\gamma^{\ast}}{\gamma}\frac{\phi_{\epsilon}}{\overline{\phi}_{\epsilon}^{\ast}}\right)\\
 \nonumber
&=\frac{\lre_2}{P_{\epsilon}^{n+1}}\left(
 \frac{\phi_{\epsilon}^{\ast}}{\overline{\phi}_{\epsilon}}
 -\frac{\phi_{\epsilon}}{\overline{\phi}_{\epsilon}^{\ast}}
  +\lre_{1,-1}^{\ast}\frac{\phi_{\epsilon}^{\ast}}{\overline{\phi}_{\epsilon}}
  +\lre_{1,-1}\frac{\phi_{\epsilon}}{\overline{\phi}_{\epsilon}^{\ast}}\right)\\
  \nonumber
&=\frac{\lre_2}{P_{\epsilon}^{n+1}}\left(
 \frac{\phi_{\epsilon}\lre_3+\overline{\phi}_{\epsilon}^{\ast}\lre_3}
 {\overline{\phi}_{\epsilon}\overline{\phi}_{\epsilon}^{\ast}}\right)+
 \frac{1}{\gamma}\lra_{(-1,1)}^{\epsilon}
 +\frac{1}{\gamma^{\ast}}\lra_{(-1,1)}^{\epsilon}\\
 \nonumber
&=\frac{1}{\gamma}\lra_{(-1,1)}^{\epsilon}
 +\frac{1}{\gamma^{\ast}}\lra_{(-1,1)}^{\epsilon}.
\end{align}
Similarly,
\begin{equation}
\label{case24}
\frac{\gamma}{\gamma^{\ast}}\frac{1}{P^n_{\epsilon}}\frac{\phi_{\epsilon}^{\ast}}{\overline{\phi}_{\epsilon}}
 -\frac{\gamma^{\ast}}{\gamma}\left(\frac{1}{P^n_{\epsilon}}\frac{\phi_{\epsilon}^{\ast}}{\overline{\phi}_{\epsilon}}\right)^{\ast}
 =\frac{1}{\gamma}\lra_{(-1,1)}^{\epsilon}
 +\frac{1}{\gamma^{\ast}}\lra_{(-1,1)}^{\epsilon}.
\end{equation}
From the last three terms in Lemma \ref{derM} $iii)$ and
(\ref{case21}), (\ref{case22}), (\ref{case23}), and (\ref{case24})
we conclude
\begin{equation*}
\gamma(\zeta)L_n\lrm_{kj}^{\epsilon0}-\gamma(z)(L_n\lrm_{jk}^{\epsilon0})^{\ast}=
\frac{1}{\gamma}\lra_{(-1,1)}^{\epsilon}
 +\frac{1}{\gamma^{\ast}}\lra_{(-1,1)}^{\epsilon}.
\end{equation*}
\newline
{\it Case 3.}  $n\notin K$ and $n\in L$.
\begin{align}
\nonumber
 A_{KL}^{\epsilon0}&=-\varepsilon_{nQ}^L
\sum_{j<n}\varepsilon_{nK}^{njQ}
(L_n\lrm_{nj}^{\epsilon0})\\
 \label{AKL3}
&=-\varepsilon_{nQ}^L\sum_{j<n}\varepsilon_{K}^{jQ}
 \overline{L}_j\rho^2\frac{4n(n-1)\phi^{\ast}_{\epsilon}}
 {\gamma\gamma^{\ast}P_{\epsilon}^{n+1}}+
 \frac{1}{\gamma^2}\lra_{(-1,1)}^{\epsilon}+
 \frac{1}{\gamma\gamma^{\ast}}\lra_{(-1,1)}^{\epsilon}
\end{align}
by Lemma \ref{derM} iv).
\newline
{\it Case 4.}  $n\in K$ and $n\notin L$.  From \cite{LiMi} (see
IV.2.57)
\begin{equation*}
A_{KL}^{\epsilon0}=\varepsilon_{nJ}^K\sum_{k\in
L}\varepsilon_{kJ}^L\sum_{j<n} L_j\lrm_{kj}^{\epsilon0} -
\sum_{m,k\in L\atop j\in J} \varepsilon_{kmM}^L\varepsilon_{njM}^K
L_m\lrm_{kj}^{\epsilon0}.
\end{equation*}
The second sum is $\frac{1}{\gamma}\lra_{(-1,1)}^{\epsilon}+
 \frac{1}{\gamma^{\ast}}\lra_{(-1,1)}^{\epsilon}+
 \frac{1}{\gamma^2}\lra_{(0,2)}^{\epsilon}$
because
\begin{equation*}
L_k\lrm_{mj}^{\epsilon0}=L_m\lrm_{kj}^{\epsilon0}+
 \frac{1}{\gamma}\lra_{(-1,1)}^{\epsilon}+
 \frac{1}{\gamma^{\ast}}\lra_{(-1,1)}^{\epsilon}+
 \frac{1}{\gamma^2}\lra_{(0,2)}^{\epsilon}
\end{equation*}
from Lemma \ref{derM} i). For the first sum we set $m=j$ in Lemma
\ref{derM} i), and use
\begin{align*}
\sum_{j<n}L_j\lrm_{kj}^{\epsilon0}=&\frac{2n(n-1)}
 {\overline{\phi}_{\epsilon}P_{\epsilon}^{n}}L_k\rho^2
 -\frac{n(n-1)}{\overline{\phi}_{\epsilon}P_{\epsilon}^{n+1}}
 L_k\rho^2\sum_{j<n}|L_j\rho^2|^2\\
 &+\frac{1}{\gamma}\lra_{(-1,1)}^{\epsilon}
 +
 \frac{1}{\gamma^{\ast}}\lra_{(-1,1)}^{\epsilon}+
 \frac{1}{\gamma^2}\lra_{(0,2)}^{\epsilon},
\end{align*}
the first two terms on the right of which can be written by
Proposition \ref{2p-l2} as
\begin{align*}
\frac{n(n-1)}{\overline{\phi}_{\epsilon}P_{\epsilon}^{n+1}}
L_k\rho^2 \left( 2P_{\epsilon} -\sum_{j<n} |L_j\rho^2|^2  \right)
 =&\frac{4n(n-1)}{\gamma\gamma^{\ast}}\frac{|\phi_{\epsilon}|^2}
 {\overline{\phi}_{\epsilon}P_{\epsilon}^{n+1}}L_k\rho^2\\
 &+
 \frac{1}{\gamma\gamma^{\ast}}\lra_{(-1,1)}^{\epsilon}
 +
 \frac{1}{\gamma^2}\lra_{(0,2)}^{\epsilon}.
\end{align*}
This gives
\begin{equation}
\label{AKL4}
 A_{KL}^{\epsilon0}=
\frac{4n(n-1)}{\gamma\gamma^{\ast}}\frac{\phi_{\epsilon}}{P_{\epsilon}^{n+1}}
 \varepsilon_{nJ}^K\sum_{k<n}\varepsilon_{kJ}^L L_k\rho^2+
  \frac{1}{\gamma\gamma^{\ast}}\lra_{(-1,1)}^{\epsilon}
  +\frac{1}{\gamma^2}\lra_{(0,2)}^{\epsilon}.
\end{equation}
{\it Case 3.} $n\notin K$ and $n\in L$. Comparing (\ref{AKL3}) and
(\ref{AKL4}) we obtain
\begin{align*}
A_{KL}^{\epsilon0}&=-\varepsilon_{nQ}^L\varepsilon_{K}^{kQ}\frac{4n(n-1)}{\gamma\gamma^{\ast}}
\frac{\phi_{\epsilon}^{\ast}}{P_{\epsilon}^{n+1}}\overline{L}_k\rho^2+
 \frac{1}{\gamma^2}\lra_{(-1,1)}^{\epsilon}+\frac{1}{\gamma\gamma^{\ast}}\lra_{(-1,1)}^{\epsilon}\\
A_{LK}^{\epsilon0}&=\varepsilon_{nJ}^L\varepsilon_{kJ}^{K}\frac{4n(n-1)}{\gamma\gamma^{\ast}}
\frac{\phi_{\epsilon}}{P_{\epsilon}^{n+1}}L_k\rho^2+
\frac{1}{\gamma\gamma^{\ast}}\lra_{(-1,1)}^{\epsilon}
 +\frac{1}{\gamma^2}\lra_{(0,2)}^{\epsilon},
\end{align*}
and
\begin{equation*}
\gamma(\zeta)A_{KL}^{\epsilon0}-\gamma(z)(A_{LK}^{\epsilon0})^{\ast}=
 \frac{1}{\gamma}\lra_{(-1,1)}^{\epsilon}+\frac{1}{\gamma^{\ast}}\lra_{(-1,1)}^{\epsilon}.
\end{equation*}
{\it Case 4.} We can reduce it to Case 3 by
\begin{equation*}
\gamma(\zeta)A_{KL}^{\epsilon 0}-\gamma(z)(A_{LK}^{\epsilon
0})^{\ast}=
 -\big(\gamma(\zeta)A_{LK}^{\epsilon 0}-\gamma(z)(A_{KL}^{\epsilon 0})^{\ast}\big)^{\ast}.
\end{equation*}
\end{proof}
This also concludes the proof of Theorem \ref{thrmP}.  As an
important corollary to theorems \ref{thrmT} and \ref{thrmP} we see
if we compose the operator ${\mathbf P}^{\epsilon}_q$ with
$\gamma^2\gamma^{\ast}$ or with $\gamma(\gamma^{\ast})^2$, we
obtain operators which are of type 1.  This is the idea behind the
next theorem which results from multiplying the basic integral
representation Theorem \ref{bir} by an appropriate number of
factors of $\gamma$ and $\gamma^{\ast}$.  We can then let
$\epsilon\rightarrow 0$ to obtain a representation on the domain
$D$.

For a given $f\in
L^2_{0,q}(D)\cap\mbox{Dom}(\mdbar)\cap\mbox{Dom}(\mdbar^{\ast})$
we take a sequence $\{f_{\epsilon}\}_{\epsilon}$ which approaches
$f$ in the graph norm by Proposition \ref{mdense}.
 With the use of Theorem \ref{bir} we define operators
$T^{\epsilon}_q$, $S^{\epsilon}_q$, and $P^{\epsilon}_q$ so that
we have the representation for each $f_{\epsilon}$
\begin{equation}
\label{tspep}
 f_{\epsilon}(z)= T^{\epsilon}_q \mdbar f_{\epsilon}
+ S^{\epsilon}_q\mdbar^{\ast}_{\epsilon}f_{\epsilon} +
P^{\epsilon}_q f_{\epsilon}.
\end{equation}
We then define the operators $T_q$, $S_q$, and $P_q$ to be such
that $\gamma^{\ast}T_q\circ\gamma^2$,
$\gamma^{\ast}S_q\circ\gamma^2$, and
 $\gamma^{\ast}P_q\circ\gamma^2$
are the limit operators, as $\epsilon\rightarrow 0$, of
 $\gamma^{\ast}T_q^{\epsilon}\circ\gamma^2$,
 $\gamma^{\ast}S_q^{\epsilon}\circ\gamma^2$, and
 $\gamma^{\ast}P_q^{\epsilon}\circ\gamma^2$, respectively,
 which exist by
Proposition \ref{typicalest}.  We therefore obtain the
\begin{thrm}
\label{basicintrep}
 For $f\in L^2_{(0,q)}(D)\cap Dom(\mdbar)\cap Dom(\mdbar^{\ast})$,
\begin{equation*}
\gamma(z)^3 f(z)=\gamma^{\ast} T_q \mdbar \left(\gamma^2 f\right)
+\gamma^{\ast} S_q\mdbar^{\ast}\left(\gamma^2 f\right)
+\gamma^{\ast} P_q\left(\gamma^2 f\right).
\end{equation*}
\end{thrm}

\section{Estimates}

We define $Z_1$ operators to be those which take the form
\begin{equation*}
Z_1=A_{(1,1)}+E_{1-2n}\circ\gamma,
\end{equation*}
and we write Theorem \ref{basicintrep} as
\begin{equation}
\label{z1cal}
 \gamma^3 f =
 Z_1\gamma^2\mdbar f +Z_1\gamma^2\mdbar^{\ast} f
  +Z_1 f.
\end{equation}

We define $Z_j$ operators to be those operators of the form
\begin{equation*}
 Z_j=\overbrace{Z_1\circ\cdots\circ Z_1}^{j\mbox{ times}}.
\end{equation*}

By applying Corollary \ref{yngcor} and Theorem \ref{E1properties}
 $n+2$ times, we have the property
\begin{equation*}
 Z_{n+2}:L^2(D)\rightarrow L^{\infty}(D).
\end{equation*}

 We now iterate (\ref{z1cal}) to get
\begin{align*}
\label{basiciterate}
 \gamma^{3j}f=&(Z_1\gamma^{3(j-1)+2}+Z_2\gamma^{3(j-2)+2}+
  \cdots+Z_{j}\gamma^2)\mdbar f \\
 \nonumber
  &+(Z_1\gamma^{3(j-1)+2}+Z_2\gamma^{3(j-2)+2}+
  \cdots+Z_{j}\gamma^2)\mdbar^{\ast}f+Z_{j}f.
\end{align*}
Then we can prove
\begin{thrm} For $f\in L^2_{0,q}(D)\cap\mbox{Dom}(\mdbar)\cap
\mbox{Dom}(\mdbar^{\ast})$, $q\ge 1$, \label{k=0}
\begin{equation*}
\|\gamma^{3(n+2)}f\|_{L^{\infty}}\lesssim \|\gamma^2\mdbar
f\|_{\infty}+\|\gamma^2\mdbar^{\ast}f\|_{\infty}+\|f\|_2 .
\end{equation*}
\end{thrm}

\end{document}